\documentclass{amsart}

\newtheorem{thm}{Theorem}[section]

\newtheorem{lem}[thm]{Lemma}
\newtheorem{prop}[thm]{Proposition}

\newtheorem{subcriterion}[thm]{Sub-criterion}
\theoremstyle{definition}
\newtheorem{defn}[thm]{Definition}
\newtheorem{example}[thm]{Example}

\theoremstyle{remark}
\newtheorem{rem}[thm]{Remark}

\numberwithin{equation}{section}
\usepackage{amsmath,amssymb,graphicx,bbm,amsthm,hyperref,mathrsfs,tikz,palatino,bm,enumitem,color,ifpdf}
\usetikzlibrary{shapes.geometric,arrows}

\begin{document}

\title[Hyperbolicity of Brunnian Links]
{Hyperbolicity of Brunnian Links}

\author{Sheng Bai}
\address{Beijing, China}
\email{barries@163.com}

\subjclass{57M25}

\keywords{Brunnian links, hyperbolic links, links with unknotted components, surface system, intuitive proofs, hyperbolic manifolds, satellite operations, JSJ-decomposition}

\begin{abstract}
We present new techniques to show hyperbolicity of links based on geometric/combinatorial topology. Our techniques are applicable to links that have at least one unknotted component. In particular, they are applicable to Brunnian links. We illustrate our techniques on many examples of Brunnian links. In fact, we determine hyperbolicity for all Brunnian links found in literature. In particular, we discover many infinite families of hyperbolic Brunnian links.
\end{abstract}

\date{Mar 1 , 2022}
\maketitle

\begin{figure}[htbp]
		  \centering
    \includegraphics[width=\textwidth]{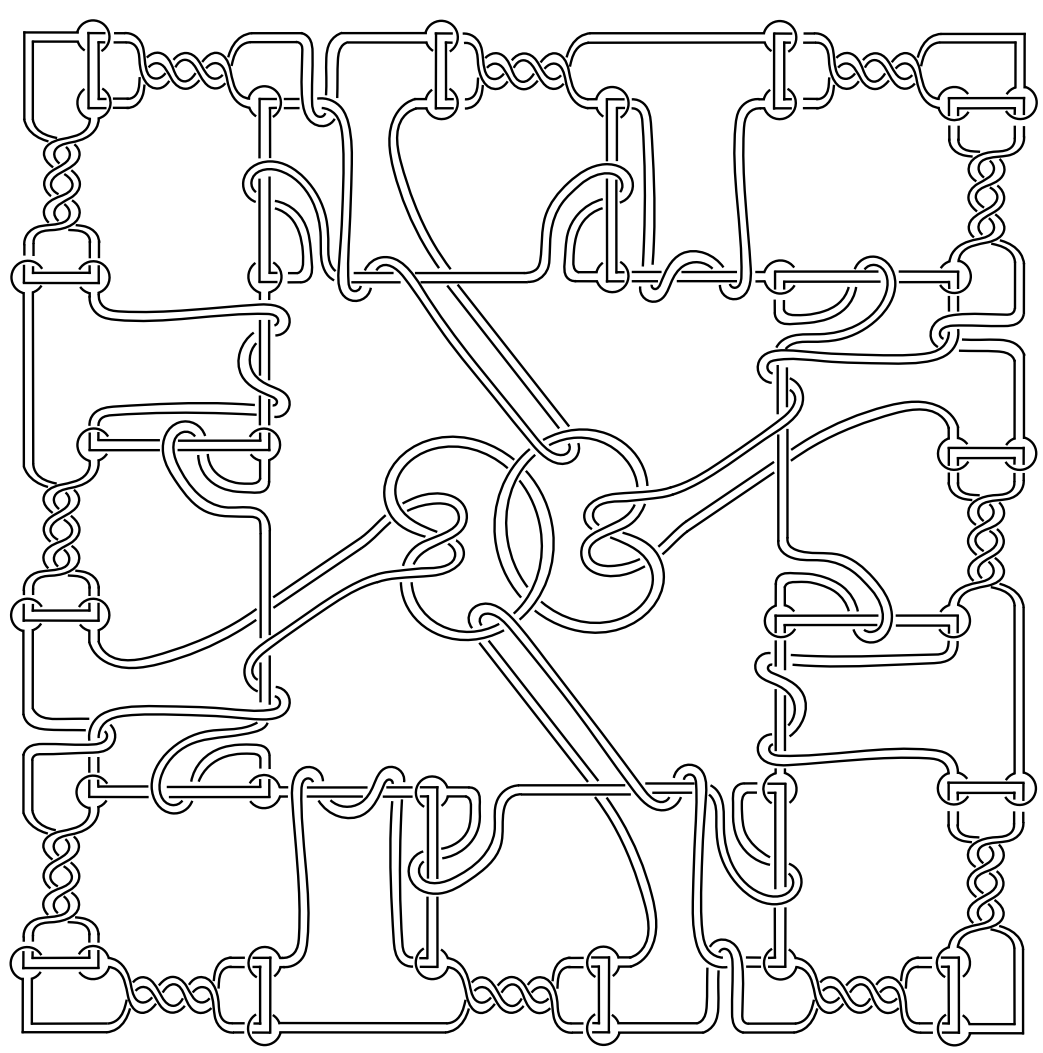}
    \caption{A hyperbolic Brunnian link: Cirrus in \cite{BW}.}
	\label{fig:cirrus}
\end{figure}

\tableofcontents

\section{Introduction}\label{sect:introduction}
Although knot invariants have been extensively developed in recent decades, there is still a lack of intuitively convenient solutions to many basic problems of knots. The study of links encompasses the study of knots, but this is not necessarily the case for links. As shown in \cite{BW}, intuitive methods, based on observation and worked by hand, have been established to show nontriviality of links with at least one unknotted component. In this paper, we explore convenient manual methods to detect hyperbolicity of links with at least one unknotted component, mainly focusing on Brunnian links. 

\subsection{Background of Brunnian links}

A link in $S^3$ is \emph{Brunnian} if it is nontrivial and its proper sublinks are all trivial. 
Brunnian links form an important class in the study of links. Analyzing a link, we may naturally consider its sublinks. The nontrivial sublinks with least components are exactly knots and Brunnian links. In this sense, knots and Brunnian links can be viewed as building blocks for general links. 

Brunnian links were introduced by H. Brunn\cite{Brunn} in 1892, but the Brunnian property of Brunn's links was not proved until 1961\cite{HDb}. Since then, many authors constructed plenty of examples, which, to our best knowledge, all appeared in \cite{Brunn, HDb, Rolfsen, BFJRVZ, BS, BCS, L, J, F, M, MM, BW}. In recent years, Baas et al. \cite{BFJRVZ, BS, BCS} constructed several infinite families of Brunnian links, and based on their inspiration, the authors of \cite{BW} constructed new Brunnian links in bulk in 2020, far more than have been previously constructed.

We summarise the main theoritical results on Brunnian links in \cite{BM} as follows. 

Firstly, two special satellite operations were investigated in \cite{BM}. A classical theorem by Alexander states that every smooth torus bounds at least one solid torus in $S^3$. We say a torus in $S^3$ is \emph{knotted} if it bounds only one solid torus in $S^3$; otherwise, \emph{unknotted}.

\begin{defn}\label{def:sum}
\cite{BM} Let $L$ and $L'$ be links in $S^3$, and $C \subset L$ and $C' \subset L'$ be oriented unknotted components. A homeomorphism $h: S^3 -int N(C') \longrightarrow N(C)$ maps the oriented meridian of $N(C')$  to the oriented longitude of $N(C)$ and maps the oriented longitude of $N(C')$  to the oriented meridian of $N(C)$, where $N(C)$ is the regular neighbourhood of $C$. Then the link $(L-C) \sqcup h(L'-C')$ is the \emph{s-sum} of $L(C)$ and $L'(C')$.
\end{defn}

For Brunnian links, we have
\begin{thm}\cite{BM}\label{thm:sum}
If a torus $T$ in $S^3$ splits a Brunnian link $L$, then $T$ is incompressible in $S^3 - L$, and $L$ is decomposed by $T$ as an s-sum of two Brunnian links.
\end{thm}

A Brunnian link, other than Hopf link, is called \emph{s-prime} (Def. 1.0.7 in \cite{BM}) if there is no unknotted essential torus in its complement space. In other words, a Brunnian link is s-prime if it is not a nontrivial s-sum.

\begin{defn}\label{def:tie}
\cite{BM} Let $L_0 \sqcup L^n$ be a link in $S^3$ where $L^n=\sqcup_{i=1}^n C^i$ is an oriented unlink, and $k_i$, $i=1,...,n$, be nontrivial knot types. Let $h: U_n =S^3 - int N(L^n) \longrightarrow S^3$ be an orientation-preserving embedding so that
\begin{align}
   S^3-h(U_n ) \cong \sqcup_{i=1}^n (S^3-N(k_i)), \nonumber
\end{align}
where $h(\partial N(C^i))$ corresponds to $\partial N(k_i)$, and the oriented meridian of $N(C^i)$ is mapped to the oriented null-homologous curve in $S^3-h(U_n)$ corresponding to the oriented longitude of $N(k_i)$. Then $L=h(L_0 )$ is an \emph{s-tie}(see Fig. \ref{fig:stie}).
\end{defn}
\begin{figure}[htbp]
		  \centering
    \includegraphics[width=\textwidth]{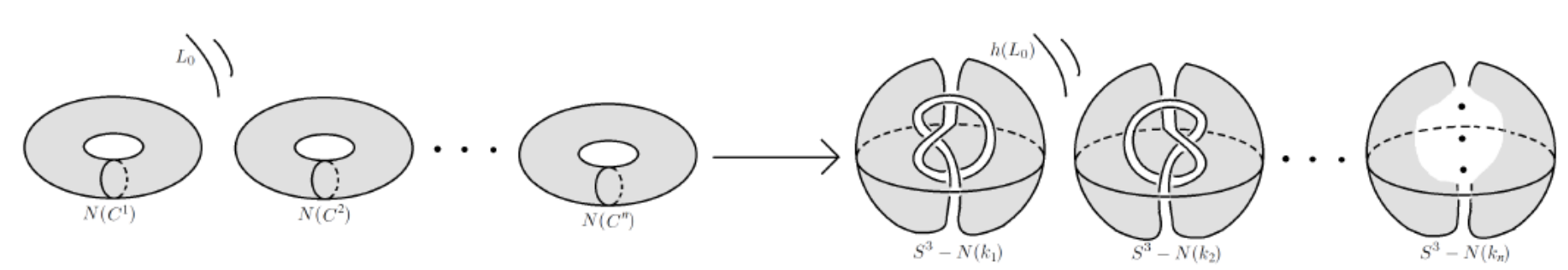}
    \caption{Construction of s-tie.}
	\label{fig:stie}
\end{figure}

For Brunnian links, we have
\begin{thm}\cite{BM}\label{thm:tie}
Every knotted essential torus in the complement of a Brunnian link bounds the whole link in the solid torus side.
\end{thm}

A Brunnian link is called \emph{untied} (Def. 1.0.8 in \cite{BM}) if there is no knotted essential torus in its complement space. In other words, a Brunnian link is untied if it is not an s-tie. 

Secondly, the authors of \cite{BM} gave the geometric classification result for Brunnian links.

\begin{thm}\cite{BM} \label{thm:classification0}
If a Brunnian link $L$ is s-prime, untied, and not a $(2, 2n)-$torus link, then $L$ has a hyperbolic structure.
\end{thm}

The authors of \cite{BM} further characterized the geometric decomposition (the JSJ-decomposition) of Brunnian links and gave a canonical decomposition for Brunnian links, simpler than the JSJ-decomposition, called \emph{tree-arrow structure}. In such representation, $(2, 2n)-$torus links, hyperbolic Brunnian links and \emph{hyperbolic Brunnian links in unlink-complements} (Sec. 2 in \cite{BM}) turned out to be the building blocks of general Brunnian links.

Lastly, the authors of \cite{BM} defined an operation to reduce a Brunnian link in an unlink-complement into a new Brunnian link in $S^3$ and thus concluded that hyperbolic Brunnian links and $(2, 2n)-$torus links are building blocks of Brunnian links in the sense that they generate all Brunnian links by s-sum and s-tie.

We refer to \cite{MS,MY,H,LWZ,BM} for more theoretic results on Brunnian links, and to \cite{BCWW} for the relationship between Brunnian links and homotopy groups of $S^2$. We note that the construction of s-tie shows that Brunnian links contain all complexity of knots (see Fig. 3 in \cite{BM}), while focusing on hyperbolic Brunnian links involves no complexity of knots. In short, by \cite{BCWW}, constructions of new Brunnian links can theoretically give all nontrivial elements in high dimensional homotopy groups of $S^2$, and by s-sum and s-tie (see \cite{BM}), a discovery of hyperbolic Brunnian link implies discovery of infinitely many Brunnian links.

\subsection{Our attacks}

A question naturally occurs: in literature, how many Brunnian links are hyperbolic? To show a Brunnian link is hyperbolic, the two main tasks are to show it is untied and it is s-prime.

In \cite{BW}, the authors provided two intuitive and generally effective methods for detecting Brunnian property of a link, if each component is known to be unknotted. We cope with the two tasks by extending the new ideas of \cite{BW}. We utilize the disks bounded by some components, instead of diagrams or triangulations. Our methods are thus manual and practical, significantly different from the known algorithmatic methods.

Firstly, in Section \ref{sect:crediblestable}, we introduce our main tool, \emph{disk system}, that is a union of the link and some spanning disks under three regularity conditions, see Def. \ref{def:spanningcomplex}. Roughly speaking, a spanning disk is \emph{stable} if it intersects some other components of the link with the least number of points (Def. \ref{def:stable}). We show that stable disks are practically obtainable. A disk system is called \emph{stable} if all the spanning disks are stable.

In Section \ref{sect:untiedness}, we give a necessary and sufficient decision theorem for a Brunnian link to be untied.
\begin{thm}\label{thm:untie0}
(Decision Theorem for Untiedness)

Let $L$ be a Brunnian link and $U$ be a stable disk system. Then $L$ is untied if and only if there is no incompressible knotted torus in the complement of $U$.
\end{thm}

To show s-primeness for Brunnian links is generally more difficult. Rather than giving one criterion, we give the following universal intermediate theorem in Section \ref{sect:simpleintersection}, which synthesizes many special cases. Then in Section \ref{sect:sprime}, we provide four subcriteria to show s-primeness based on it.

\begin{thm}\label{thm:simplepattern0}
(Simple Intersection Pattern Theorem)

Let $L$ be a Brunnian link, let $T$ be an essential torus splitting $L$, and let $U$ be a stable disk system where the cross disks are mutually disjoint. Then after an isotopy of $T$, every disk intersects $T$ in simple intersection pattern.
\end{thm}

Here by ``simple intersection pattern", we mean that the intersection between the torus and disks has a very simple form, for details, see Def. \ref{defn:simple intersection pattern}. For the meaning of \emph{cross disk}, see Def. \ref{defn:interior and cross disks}. 

In fact, the assumptions of the two corresponding theorems in the text are much weaker than these presented here, and so the conclusions in the text are stronger, but the articulation is a little more complicated.

\subsection{Hyperbolic Brunnian links}
Our methods work to detect hyperbolicity for all of Brunnian links in literature. Especially, we may first list 6 infinite families of Brunnian links:

\begin{enumerate}
  \item deBrunner($n$), where $n > 1$, as shown in Fig. \ref{fig:debrunner(n)}. (c.f. \cite{Rolfsen})
  \item W($n$), where $n > 1, n \ne 4$, as shown in Fig. \ref{fig:W5}. (c.f. \cite{Brunn, Rolfsen})
  \item Torus($m,n$), where $m,n\in \mathbb{N}_+$, as shown in Fig. \ref{fig:torus(m,n)}. (c.f. \cite{BCS, BS})
  \item Tube($m,n$), where $m>0,n>1$, as shown in Fig. \ref{fig:tube(m,n)}.(\cite{BCS, BS, BW})
  \item Carpet($m,n,p$), where $m,n,p \in \mathbb{N}_+$, as shown in Fig. \ref{fig:carpet(m,n,p)}.(\cite{BCS, BS, BW})
  \item Lamp($n_1, n_2,...,n_{2k}$), where $k \in \mathbb{N}_+$ and each index is an odd integer indicating the number of half twists as shown in Fig. \ref{fig:lamp}.
\end{enumerate}

We point out that the second family W($n$) was constructed by H. Brunn more than 120 years ago, and W($4$) is not hyperbolic as it is a Milnor link (see Fig. 2 in \cite{BM}).

\begin{figure}[htbp]
\centering
\begin{minipage}[t]{0.48\textwidth}
\centering
\includegraphics[scale=0.3]{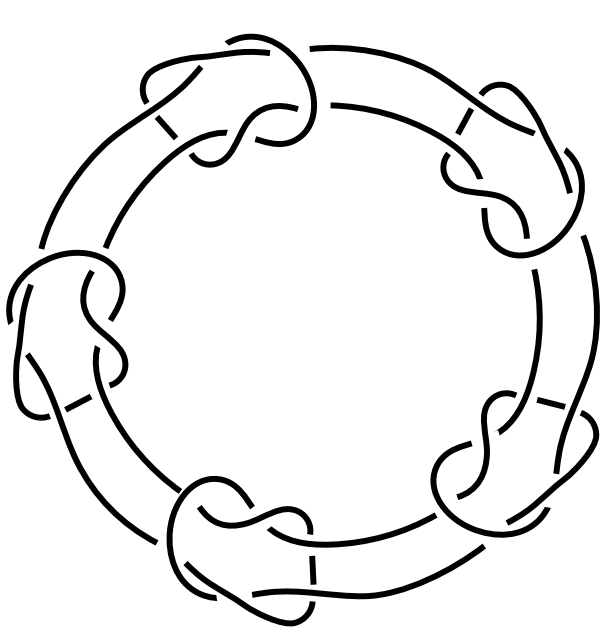}
    \caption{deBrunner(5).}
	\label{fig:debrunner(n)}
\end{minipage}
\begin{minipage}[t]{0.48\textwidth}
\centering
\includegraphics[scale=0.3]{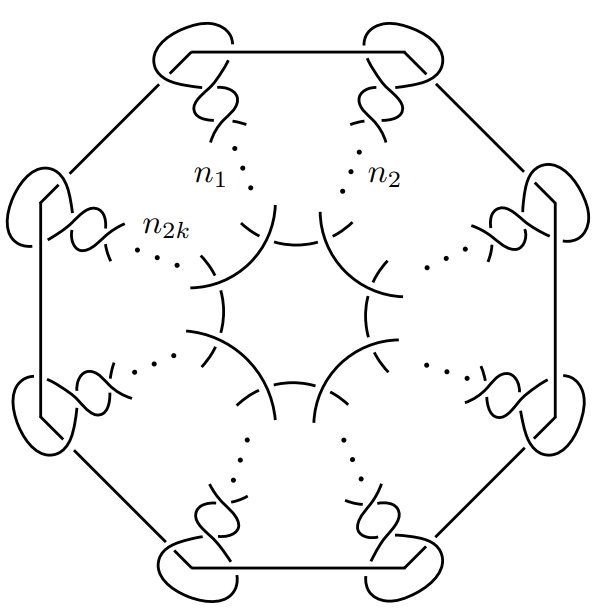}
\caption{Lamp($n_1, n_2,...,n_{2k}$).}
\label{fig:lamp}
\end{minipage}
\end{figure}

\begin{figure}[htbp]
		  \centering
    \includegraphics[width=\textwidth]{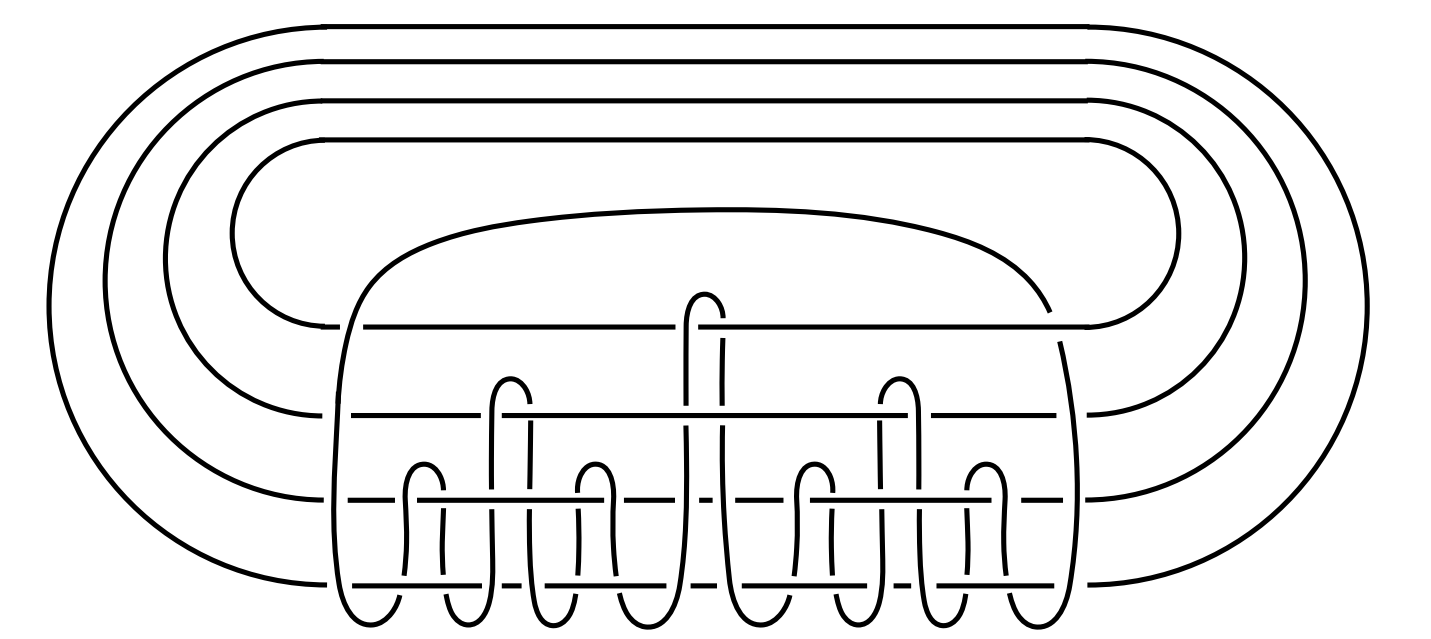}
    \caption{W(5).}
	\label{fig:W5}
\end{figure}

\begin{figure}[htbp]
		  \centering
    \includegraphics[width=\textwidth]{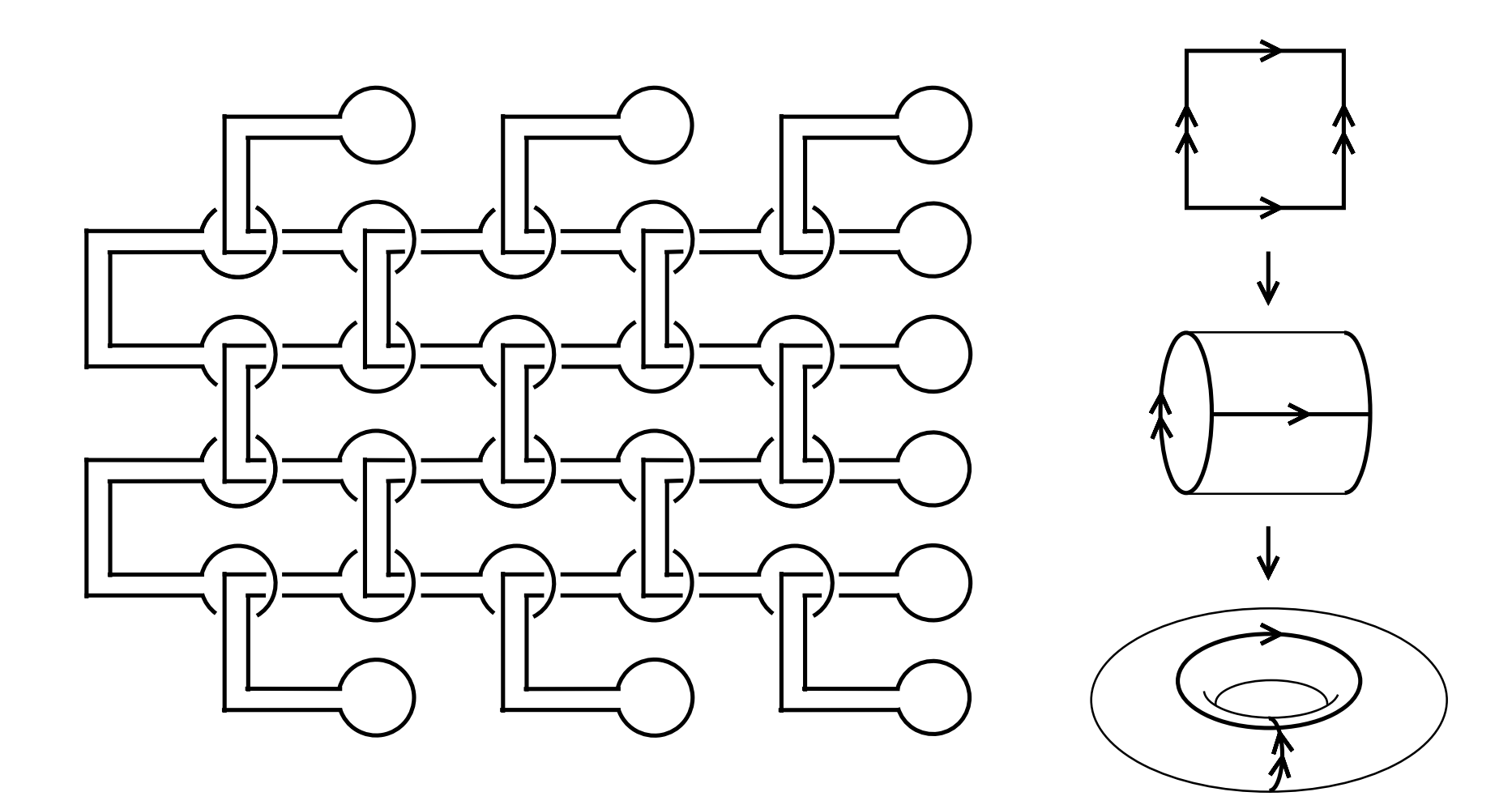}
    \caption{Torus($m,n$) has $2m$ rows and each row has $n$ components.}
	\label{fig:torus(m,n)}
\end{figure}

\begin{figure}[htbp]
		  \centering
    \includegraphics[width=\textwidth]{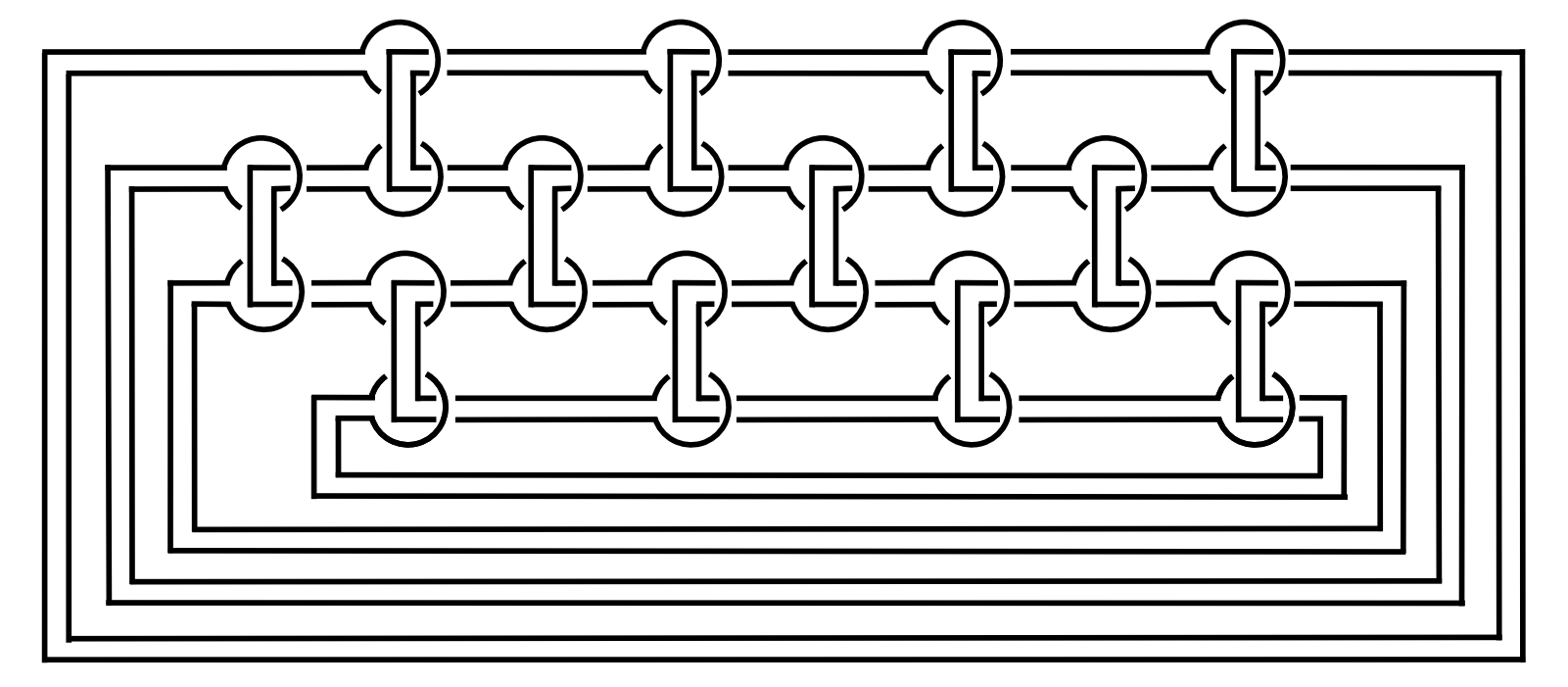}
    \caption{Tube(3,4), a tube with 3 rows and 4 columns.}
	\label{fig:tube(m,n)}
\end{figure}

\begin{figure}[htbp]
		  \centering
    \includegraphics[width=\textwidth]{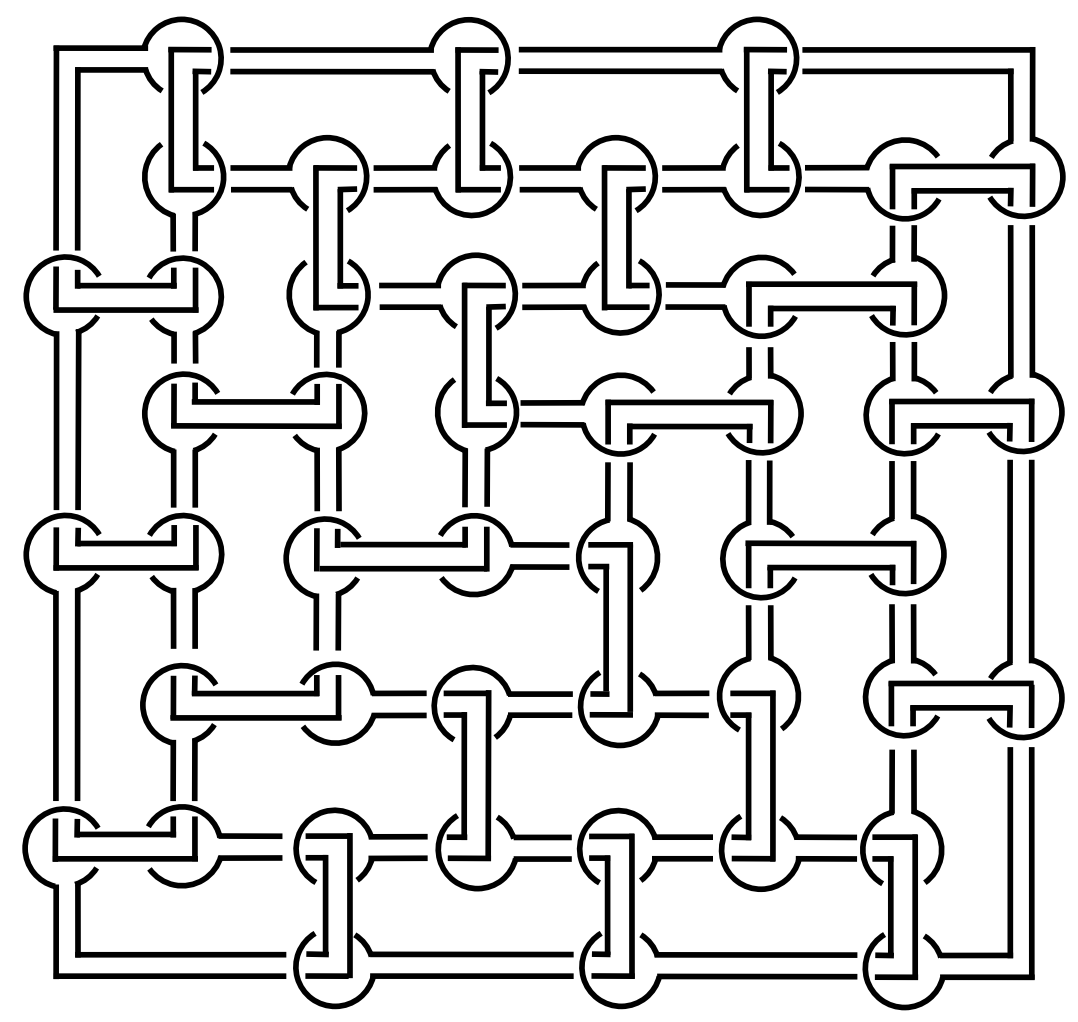}
    \caption{Carpet(1,3,4), see \cite{BW} for notations.}
	\label{fig:carpet(m,n,p)}
\end{figure}
\begin{thm}\label{thm:series}
The six families of links above are hyperbolic.
\end{thm}

The proof is given in Subsection \ref{subsect:examples}. We actually prove hyperbolicity for far more Brunnian links, some complicated families including Baas' solids (Fig. 16 in \cite{BCS}), Snakes, Cirrus and Wheels in \cite{BW}, and so on.

Comparing with using SnapPy, we see that, firstly, a program cannot show hyperbolicity for infinite families of links, and secondly, we show their hyperbolicity using flexible manual methods, thus if some details of these links are changed, it is easy to show the new links are still hyperbolic using the almost same proofs. By the theorem of W. Menasco\cite{Men}, it can be verified that alternating Brunnian links are hyperbolic. These newly discovered families are far from alternating at least seen from their diagrams. Besides, our methods provide alternative proofs for hyperbolicity of alternating Brunnian links.

In Subsection \ref{subsect: generallinks}, we generalize our criteria and framework to detect hyperbolicity for links with some unknotted components. These criteria are Theorem \ref{thm:unsplit}, \ref{thm:split} and \ref{thm:simplepattern2}. For examples, many links, not Brunnian, in Thurston's book \cite{Th}, can be fast proven hyperbolic using our methods.

\subsection{Preliminaries}\label{sect:preliminaries}
In this paper, all objects and maps are smooth. Without special explanation, we always consider links in $S^3$. All intersections are compact and transverse. Each arc is assumed to be simple. The notations $C,L,D$ (maybe with subscripts) are always used to denote an unknot, a link and an embedded disk respectively. We use $N(\cdot)$ to denote a regular neighbourhood.

For a compact set $K$ in a solid torus, $K$ is \emph{geometrically essential} in it if any meridian disk intersects $K$. If $V$ is a regular neighbourhood of a knot $K$ in $S^3$, the \emph{longitude} of $V$ is the essential curve in $\partial V$ that is null-homologous in $S^3-intV$, oriented similarly to $K$; the \emph{meridian} is the essential curve in $\partial V$  bounding a disk in $V$ and having linking number $+1$ with $K$. (c.f. \cite{RB})

\section{Credible disks and stable disks}\label{sect:crediblestable}

In this section, we introduce some notions for links with some unknotted components and our main tools, and explain these tools are always available.

\subsection{Credible disks}\label{subsect:credible}

Our basic tool is credible disks.

\begin{defn}\label{def:credible}
Let $C_i$ be a component of a link $L$ and $D_i$ be a disk bounded by $C_i$. $D_i$ is \emph{credible} if there is no circle $C$ in the interior of $D_i$ such that
\begin{itemize}
  \item[(i)]  the disk contained in $D_i$ and bounded by $C$ intersects $L$;
  \item[(ii)] there is a disk $D_C$ bounded by $C$ with $D_C \subset S^3 - L$.
\end{itemize}
\end{defn}

To simplify the task of showing a disk is credible, we give a definition used for the proposition following it.
\begin{defn}\label{def:incredible}
Let $L=\cup_{i=1}^n C_i$ be an $n$-component link, let $I$ be a subset of $\{1,...,n\}$, and let  $D_i (i\in I)$ be disjoint disks bounded by $C_i (i\in I)$. Set $U=\cup_{i \in I} D_i \cup L$. A circle $C$ in the interior of $D_i$ is \emph{incredible} for $U$ if
\begin{itemize}
  \item[(i)]  the disk contained in $D_i$ and bounded by $C$ intersects $L$;
  \item[(ii)] there is a disk $D_C$ bounded by $C$ with $int D_C \subset S^3 - U$.
\end{itemize}
\end{defn}

\begin{thm}\label{prop:credible}
Let $L=\cup_{i=1}^n C_i$ be an $n$-component link, let $I$ be a subset of $\{1,...,n\}$, and let  $D_i (i\in I)$ be disjoint disks bounded by $C_i (i\in I)$. Set $U=\cup_{i \in I} D_i \cup L$. Then $D_i$ is credible for each $i \in I$ if and only if there is no incredible circle for $U$.
\end{thm}

\begin{proof}
Since the``only if" implication is obvious, we need only prove the ``if" implication. Suppose for an $i \in I$, there is a circle $C \subset int D_i$ such that the disk contained in $D_i$ and bounded by $C$ intersects $L$ and there is a disk $D_C$ bounded by $C$ with $D_C \subset S^3 - L$. We wish to find an incredible circle for $U$.

After a perturbation near $C$, we may assume $D_C \cap \cup_{i \in I} D_i$ is a disjoint union of circles. Choose an innermost circle in $D_C$, which bounds a disk, say $D_0$, in it. Suppose $C_0= \partial D_0$ is in $D_j$ for some $j \in I$. If $C_0$ encloses some intersection points with $L$ in $D_j$, then $C_0$ is an incredible circle for $U$. Otherwise, $D_0$ caps a 3-ball with $D_j$. Replacing $D_0$ in $D_C$ by the disk contained in $D_j$ and bounded by $C_0$, a little beyond, we eliminate $C_0$. Step by step, we find an incredible circle for $U$.
\end{proof}

Noting that the complement of $U$ is smaller than the component of $L$, this theorem reduces the task of showing a disk is credible in applications. In fact, the main content of \cite{BW} is to show disks are credible, which proposes two widely applicable methods, including several criteria and a general procedure.

\subsection{Stable disks}\label{subsect:stable}
Sometimes we will use a special kind of credible disks, called stable disks. Let $C_i$ be a component of a link $L$ and $D_i$ be a disk bounded by $C_i$. Let $L_0$ consist of the components of $L$ which intersects the interior of $D_i$, and $L_J$ be a sublink of $L - C_i$ containing $L_0$.

\begin{defn}\label{def:stable}
The disk $D_i$ is \emph{stable for $L_J$} if $\sharp (D_i \cap L_J)$ is minimal among all disks bounded by $C_i$ whose interior only intersects $L_J$, and $D_i$ is \emph{stable} if it is stable for $L_0$.
\end{defn}

Clearly a stable disk is credible and the converse is not true in general. To reduce the task of showing a disk is stable, we introduce the following definition.

\begin{defn}\label{def:unstablefor}
A circle $C$ in the interior of $D_i$ is \emph{unstable for $L_J$} if $C$ bounds a disk $D_{C1}$ such that
\begin{itemize}
  \item[(i)]  $D_{C1} \cap L = D_{C1} \cap L_J$;
  \item[(ii)] $\sharp (D_{C1} \cap L) < \sharp (D_{C0} \cap L)$, where $D_{C0}$ is the disk contained in $D_i$ and bounded by $C$.
\end{itemize}
\end{defn}

Obviously that $D_i$ is stable for $L_J$ implies that there is no unstable circle for $L_J$. The following lemma shows the converse is also true.

\begin{lem}\label{lem:stabledisk}
There is no unstable circle for $L_J$ in $D_i$ if and only if $D_i$ is stable for $L_J$.
\end{lem}

\begin{proof}
Suppose there is no unstable circle for $L_J$ on $D_i$. Assume for contradiction that $D'_i$ is a disk bounded by $C_i$, having less number of intersection points with $L_J$ than $D_i$. After perturbation, $D'_i$ intersects $D_i$ transversely in their interiors. Choose an innermost circle in $D'_i$. Using the same argument as in the proof of Lemma \ref{lem:stabledisks}(2), we get a contradiction.
\end{proof}

The task of showing a disk is stable can be further reduced by considering many mutually disjoint disks. Let $L=\cup_{i=1}^n C_i$ be an $n$-component link and $I \subset \{1,...,n\}$. Take disjoint disks $D_i (i\in I)$ bounded by $C_i (i\in I)$.

\begin{defn}\label{def:unstableavoiding}
For any $k \in I$, a circle $C$ in the interior of $D_k$ is \emph{unstable avoiding $\cup_{i \in I} D_i$} if $C$ bounds a disk $D_{C1}$ such that
\begin{itemize}
  \item[(i)]  $D_{C1} \cap \cup_{i \in I} D_i = C$;
  \item[(ii)] $\sharp (D_{C1} \cap L) < \sharp (D_{C0} \cap L)$, where $D_{C0}$ is the disk contained in $D_i$ and bounded by $C$.
\end{itemize}
\end{defn}

\begin{lem}\label{lem:stabledisks}
(1) Fix $k \in I$. Then there is no unstable circle in $D_k$ avoiding $\cup_{i \in I} D_i$ if and only if there is no disk $D$ bounded by $C_k$ such that
\begin{itemize}
  \item[(i)]  $D \cap \cup_{i \in I} D_i = C_k$, and
  \item[(ii)] $\sharp (D \cap (L-C_k)) < \sharp (D_k \cap (L-C_k))$.
\end{itemize}

(2)  For any $k \in I$ there is no unstable circle in $D_k$ avoiding $\cup_{i \in I} D_i$ if and only if for any $k \in I$ there is no unstable circle in $D_k$ for $L - \cup_{i \in I} C_i$.
\end{lem}

\begin{proof}
(1). The``only if" implication is trivial. For the ``if" implication, suppose there is an unstable circle $C$ bounding $D_{C0}$ in $D_k$ and $D_{C1}$ as in Def. \ref{def:unstableavoiding}. Then the disk $(D_k - D_{C0}) \cup D_{C1}$ has less intersection points with $L-C_k$, and after perturbation, its interior is disjoint from $D_k$.

(2). Since the``if" implication is obvious, we need only prove the ``only if" implication. Set $L_J = L - \cup_{i \in I} C_i$. Suppose there is a circle $C$ in the interior of some $D_k$, bounding $D_{C0}$ in $D_k$ and bounding $D_{C1}$ as in Def. \ref{def:unstablefor}. Then $D_{C1} \cap \cup_{i \in I} D_i$ is a disjoint union of circles. Let $C_0$ be an innermost circle in $D_{C1}$, which bounds a disk $D_0$ in it and bounds a disk $D_{j0}$ in some $D_j$. Since there is no unstable circle in $D_j$ avoiding $\cup_{i \in I} D_i$, $D_{j0}$ has no more number of intersection points with $L_J$ than $D_0$. Take the immersed disk $(D_{C1} - D_0) \cup D_{j0}$ and push the $D_{j0}$ in it a little downward $D_j$ to eliminate $C_0$. We denote the obtained immersed disk by $E_{C2}$.

\begin{figure}[htbp]
		  \centering
    \includegraphics[width=0.90\textwidth]{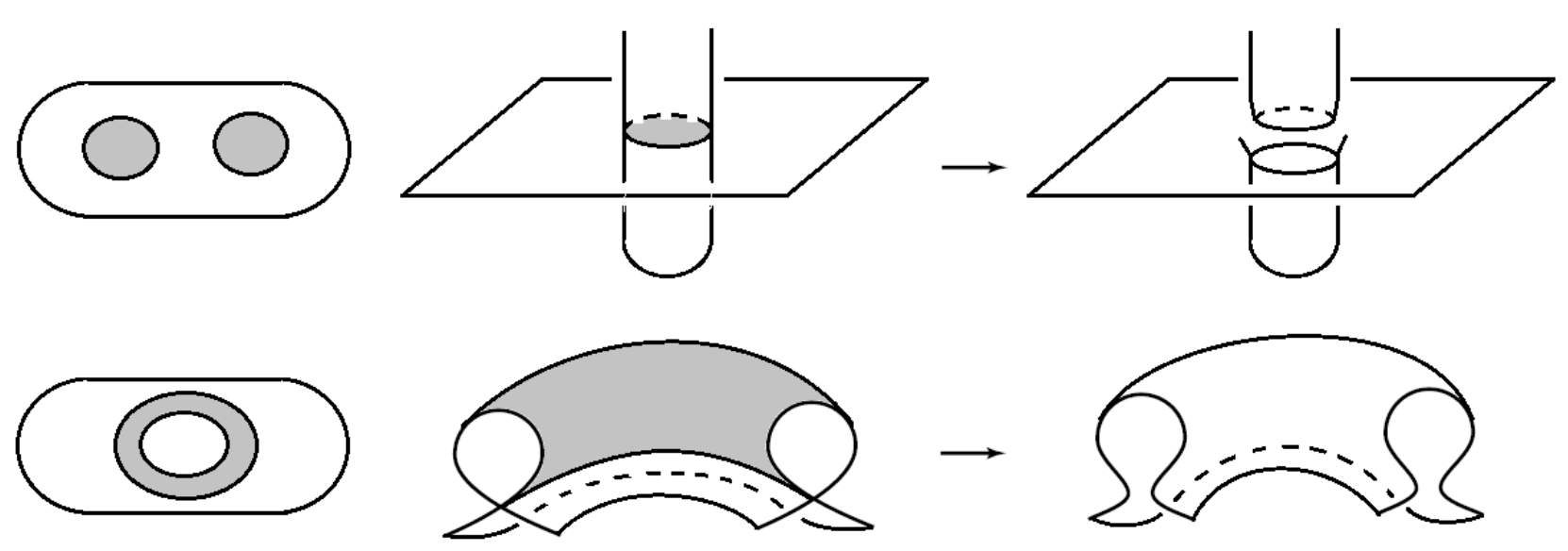}
    \caption{Two kinds of surgeries to eliminate double curves.}
	\label{fig:immersion}
\end{figure}

Generally $E_{C2}$ may not be embedded since $C_0$ may not be innermost in $D_j$. Nevertheless, its singularities are all double points, forming disjoint circles. This is an easy case of Dehn's Lemma(\cite{Pp, JM}), so we can revise it to get an embedding disk. For each self-intersection circle, change the immersion map either on two disks or in an annulus region, as illustrated in Fig. \ref{fig:immersion}, and then smooth the corners. We will get an embedding disk $D_{C2}$ having the same intersection points with $L_J$ as $E_{C2}$.

Step by step we eventually get a disk $D_{CN}$, which intersects $\cup_{i \in I} D_i$ in $C$ and has not more intersection points with $L_J$ than $D_{C1}$. Then $D_{CN}$ has less number of intersection points with $L_J$ than $D_{C0}$. Thus $C$ is unstable avoiding $\cup_{i \in I} D_i$, a contradiction.
\end{proof}

In summary, the following theorem simplifies significantly the task of showing a disk is stable.

\begin{thm}\label{prop: stable}
Let $L=\cup_{i=1}^n C_i$ be an $n$-component link, let $I$ be a subset of $\{1,...,n\}$ and let $D_i (i\in I)$ be disjoint disks bounded by $C_i (i\in I)$. Then every $D_i (i\in I)$ is stable for $L - \cup_{i \in I} C_i$ if and only if for any $i \in I$ there is no disk $D$ bounded by $C_i$ such that
\begin{itemize}
  \item[(i)]  $D \cap \cup_{k \in I} D_k = C_i$;
  \item[(ii)] $\sharp (D \cap (L-C_i)) < \sharp (D_i \cap (L-C_i))$.
\end{itemize}
\end{thm}

\begin{proof}
This follows immediately from the previous two lemmas.
\end{proof}

Applying this theorem, we present how to prove a disk is stable by two examples.

\begin{example}\label{example:Wnstable}
Consider the Milnor link in Fig. \ref{fig:g1g2g3}(1). The components $C_1, C_2, C_3$ bound
mutually disjoint oriented disks $D_1,D_2, D_3$ as Fig. \ref{fig:g1g2g3}(2), and the component $C_0$ is homotopically nontrivial in $S^3 -\cup_{i=1}^3 C_i$, whose fundamental group is freely generated by $g_1, g_2, g_3$ as Fig. \ref{fig:g1g2g3}(2). To obtain the representation of a loop in $\pi _1 (S^3 -\cup_{i=1}^3 C_i )$, just record the intersections of the loop and the disks in order with orientations. Notice that $C_0$ intersects $D_1$ and $D_2$ both in 4 points, and intersects $D_3$ in two points. With the base point $P \in C_0$, we see $C_0$ represents $[[g_1,g_2],g_3]$ in $\pi _1 ( S^3 -\cup_{i=1}^3 C_i )$. We show $D_1,D_2, D_3$ are all stable.

\begin{figure}[htbp]
		  \centering
    \includegraphics[width=0.90\textwidth]{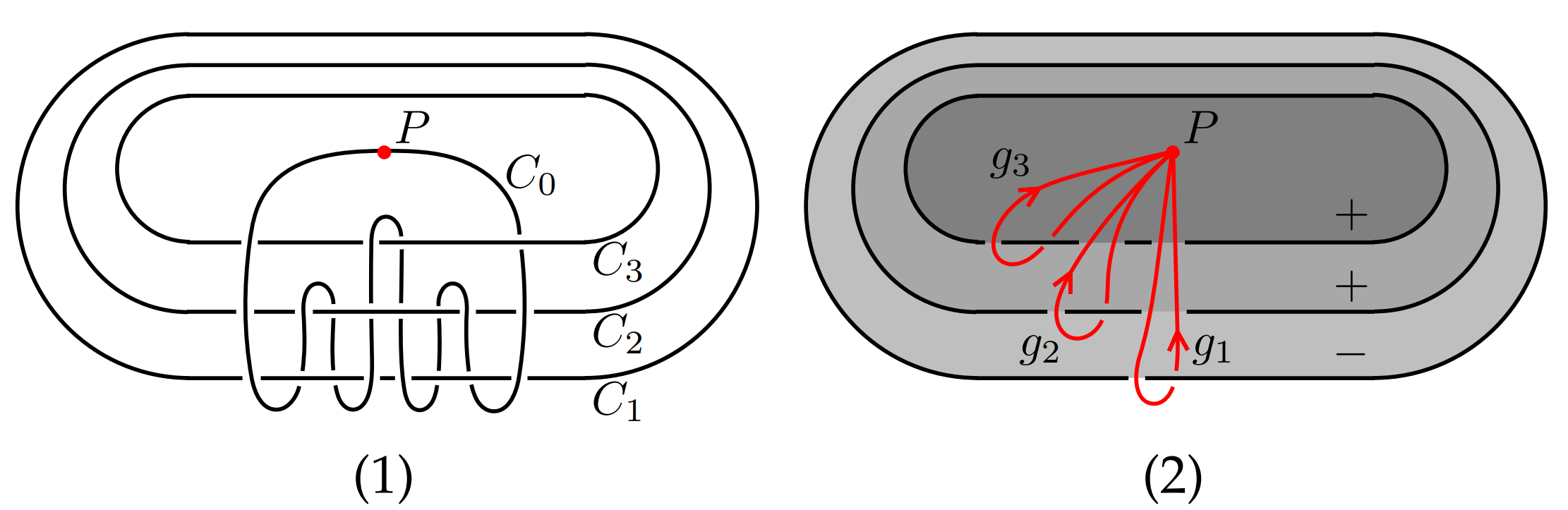}
    \caption{Stable disks for Milnor link.}
	\label{fig:g1g2g3}
\end{figure}

\begin{proof}
Since $[[g_1,g_2],g_3]$ is a reduced word, unique when generators of the free group are fixed, any isotopy of $D_1 \cup D_2 \cup D_3$ cannot give less intersection points with $C_0$ for each disk. To show $D_1$ is stable, it suffices to show there is no disk bounded by $C_1$, whose interior disjoint from $D_1 \cup D_2 \cup D_3$, has fewer intersection points with $C_0$ than $D_1$. Suppose such a disk
$D'_1$ exists, then $S_1 = D_1 \cup D'_1$ is a sphere separating $D_2$ and $D_3$. There are two cases.

{\sc Case~1:} The positive side of $D_1$ and $C_2$ belong to one side of $S_1$, which we call the positive side of $S_1$. Consider the element $C_0$ represents in $\pi _1 (S^3 -\cup_{i=1}^3 C_i )$, that is $g_1 g_2 g_1^{-1} g_2^{-1} g_3 g_2 g_1 g_2^{-1} g_1^{-1} g_3^{-1}$. We can assume the base point $P$ belongs to the negative side of $S_1$ and find an arc from $P$ representing $g_1 g_2 g_1^{-1}$, which avoids $D'_1$. Then the fourth letter $g_2^{-1}$ means we have to go to the positive side of $S_1$ to have an arc representing $g_2^{-1}$. But the third letter $g_1^{-1}$ means the arc representing the fourth letter $g_2^{-1}$ avoids $D_1$, so it must intersect $D'_1$. In other words, between the third letter and the fourth letter, $C_0$ intersects $D'_1$. In the same way, between the fourth letter $g_2^{-1} $ and the fifth letter $g_3$,   between the fifth letter $g_3$ and the sixth letter $g_2$, and between the sixth letter $g_2$ and the seventh letter $g_1$, our loop intersects $D'_1$.

In short, in the representation $g_1 g_2 g_1^{-1} g_2^{-1} g_3 g_2 g_1 g_2^{-1} g_1^{-1} g_3^{-1}$, replacing $g_2$ and $g_2^{-1}$ by $2$, $g_3$ and $g_3^{-1}$ by $3$, $g_1$ by $1^{-} 1^{+}$, and
$g_1^{-1}$ by $1^{+} 1^{-}$, we get a cyclic sequence

\begin{center}
 $1^{-} 1^{+} 2 1^{+} 1^{-} \mid 2 \mid 3 \mid 2 \mid 1^{-} 1^{+} 2 1^{+} 1^{-} 3$.
\end{center}

Since $1^{+}, 2$ belong to one side, and $1^{-}, 3$ belong to the other side, the number of
how many adjacent pairs in this cyclic sequence cross through $S_1$ gives the minimum number
of $D'_1 \cap C_0$, which is 4.

{\sc Case~2:} The positive side of $D_1$ and $C_3$ belong to one side of $S_1$. Make replacement in the representation $g_1 g_2 g_1^{-1} g_2^{-1} g_3 g_2 g_1 g_2^{-1} g_1^{-1} g_3^{-1}$ as before. Now $1^{+}, 3$ belong to one side, and $1^{-}, 2$ belong to the other side. The number of adjacent pairs in this cyclic sequence crossing through $S_1$
is 8, as

\begin{center}
$1^{-} 1^{+} \mid 2 \mid 1^{+} 1^{-} 2 \mid 3 \mid 2 1^{-} 1^{+} \mid 2 \mid 1^{+} 1^{-} \mid 3 \mid$.
\end{center}

The same method applies to $D_2, D_3$.
\end{proof}

\end{example}

\begin{figure}[htbp]
		  \centering
    \includegraphics[width=0.90\textwidth]{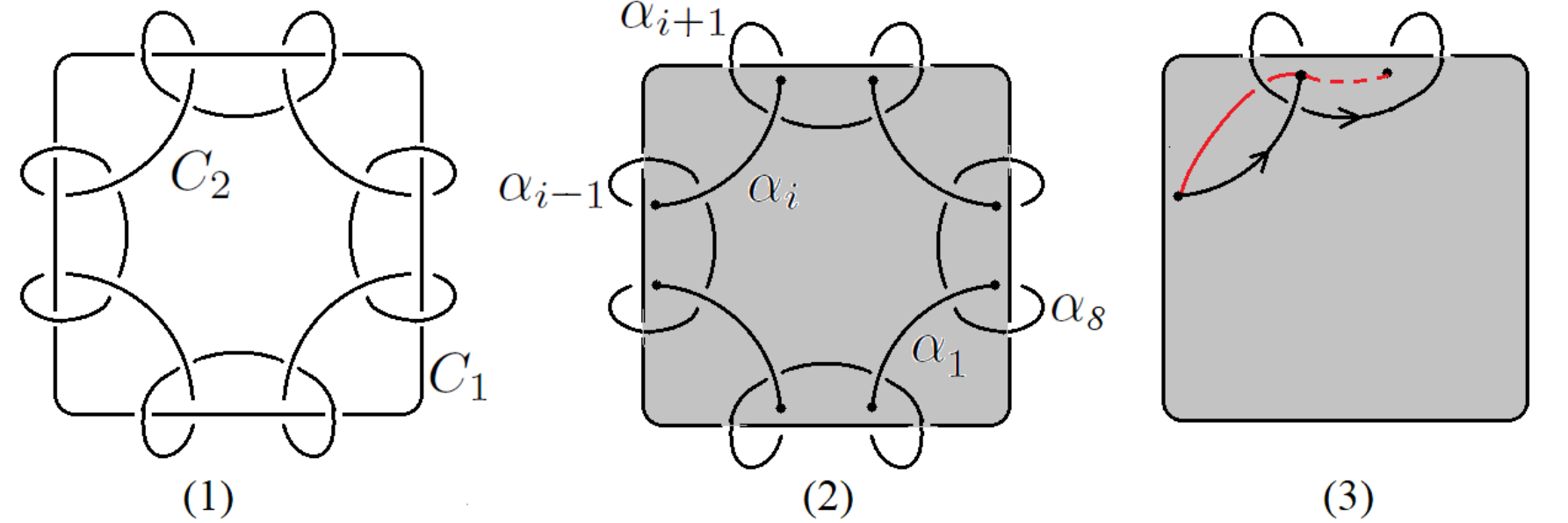}
    \caption{Lamp(1,1,1,1,1,1,1,1).}
    \label{fig:lampstable}
\end{figure}

\begin{example}\label{example:lampstable}
Consider the link in Fig. \ref{fig:lampstable}(1). We show the grey disk $D$, bounded by $C_1$ in Fig. \ref{fig:lampstable}(2) is stable. Let $D'$ be a disk bounded by $C_1$ whose interior avoids $D$. It suffices to show $\sharp (D' \cap C_2) \geq \sharp (D \cap C_2)$. The disk $D$ cuts $C_2$ into 8 arcs $\alpha_1, \alpha_2,...,\alpha_8$, successively indexed along $C_2$ as shown in Fig. \ref{fig:lampstable}(2). Notice that each $\alpha_i$ can only intersect $D'$ in an even number of points, since the two endpoints of each $\alpha_i$ are on the same side of $D$. The result will be proved by showing that for any odd $i$, either $\sharp (D' \cap \alpha _i ) \geq 2$ or $\sharp (D' \cap \alpha _{i \pm 1}) \geq 2$.

In fact, we can say $lk(\alpha _i, \alpha _{i \pm 1})=-1$ in the sense that connecting the endpoints of $\alpha_i$ by an arc on the positive side of $D$ and connecting the endpoints of $\alpha_{i \pm 1}$ by an arc on the negative side of $D$ gives two curves with linking number equal to $ -1 $. See Fig. \ref{fig:lampstable}(3). If $D'$ intersects neither $\alpha_i$ nor $\alpha_{i \pm 1}$, it would cap off $\alpha_i$ from $\alpha_{i \pm 1}$ and thus $lk(\alpha _i, \alpha _{i \pm 1})=0$, a contradiction.
\end{example}

\subsection{Disk system}\label{subsect:spanningcomplex}
We now introduce our main tool, which is a special case of the notion \emph{surface system} (c. f. \cite{C}).

\begin{defn}\label{def:spanningcomplex}
Let $L=\cup_{i=1}^n C_i$ be an $n$-component link, let $I$ be a subset of $\{1,...,n\}$, and for any $i\in I$, let $D_i$ be a credible disk bounded by $C_i$. The union $U=\cup_{i \in I} D_i \cup L$ is a \emph{disk system} for $L$ if the following regularity conditions hold:
\begin{itemize}
  \item[(Ri)]  $U$ has only generic singularities of double or triple points.
  \item[(Rii)] (No trivial intersection circle.) $\forall i,j \in I$, there is no circle component $C$ of $D_i \cap D_j$ such that the disk bounded by $C$ in $D_i$ does not intersect $L$.
  \item[(Riii)] (No trivial annulus region.) $\forall i,j \in I$, there is no pair of circle components of $D_i \cap D_j$ such that they bound annuli in both $D_i$ and $D_j$ and neither annulus intersects $L$.
\end{itemize}
\end{defn}

We point out that, given the credible disks $D_i (i\in I)$, one can always modify $\cup_{i \in I} D_i \cup L$ to meet the above regularity conditions. In fact, general differential topology guarantees condition (Ri). For condition (Rii), since $D_j$ is also credible, if such $C$ exists, then the disk bounded in $D_j$ by $C$ also does not intersect $L$. Suppose $C$ bounds $D_{Ci}$ and $D_{Cj}$ in $D_i$ and $D_j$ respectively. Interchanging $D_{Ci}$ by $D_{Cj}$ and then smoothing the corners near $C$, we can eliminate $C$. The trouble is that the new $D_i$ and $D_j$ may have self-intersection. Notice that the singularities can only be double points. By changing the immersion map and smoothing the corners, as in the second paragraph of the proof of Lemma \ref{lem:stabledisks}(2), we get embedded ones. For condition (Riii), if such pair of intersection circles exists, say $C_{01} \cup C_{02}$, which bounds $A_i$ in $D_i$ and bounds $A_j$ in $D_j$. Interchanging $A_j$ and $A_i$ and then smoothing the corners, we can eliminate $C_1 \cup C_2$. The new $D_i$ and $D_j$ may have self-intersection, but we can get embedded new $D_i$ and $D_j$ from them as above.

We conclude this section with a definition used in the statements of our theorems.

\begin{defn}\label{def:sn}
Let $C_i$ be a component of a link $L$ and $D_i$ be a credible disk bounded by $C_i$. For a natural number $N$, we say $D_i$ is \emph{s$\vee N$} if either
\begin{itemize}
  \item[(i)]  $\sharp (D_i \cap (L - C_i)) < N$; or
  \item[(ii)] $D_i$ is stable.
\end{itemize}
\end{defn}

\section{Criterion for untiedness}\label{sect:untiedness}

In this section we give a sufficient and necessary criterion to detect untiedness for Brunnian links and illustrate our method by example.

\subsection{Decision theorem and example analysis}\label{subsect:untiednessthm}

\begin{thm}\label{thm:untie}
(Untiedness Criterion)

Let $L=\cup_{i=1}^n C_i$ be an $n$-component Brunnian link, let $k$ be a number in $\{1,...,n\}$, and  let $D_i$ be a disk bounded by $C_i$ for each $i \le k$ so that $U = \bigcup_{i=1}^k D_i \cup L$ is a disk system. 

Assume that for every $i \le k$,
$$\left\{
  \begin{array}{ll}
    D_i \ is \ s \vee 7 , & if \  n=2 \  and \ lk(C_1, C_2)= \pm 1; \\
    D_i \ is \ s \vee 8, & in \ other \ cases.
  \end{array}
\right.$$

Then $L$ is untied if and only if there is no incompressible knotted torus in the complement of $U$.
\end{thm}

The proof of this theorem will be given at the end of this section.

Recall that a disk system is always available. The ``if'' implication of this theorem gives a method to show a Brunnian link is untied. The ``only if'' implication indicates this method is theoretically universal. Furthermore, if a Brunnian link is an s-tie, we can always find a companion torus by choosing a disk system $U$ satisfying the condition in Theorem \ref{thm:untie}, since an incompressible knotted torus in $S^3 - U$ is a desired companion.

Let us illustrate our method with an example.

\begin{example}\label{example:brunnlink}

\begin{figure}[htbp]
		  \centering
    \includegraphics[width=0.80\textwidth]{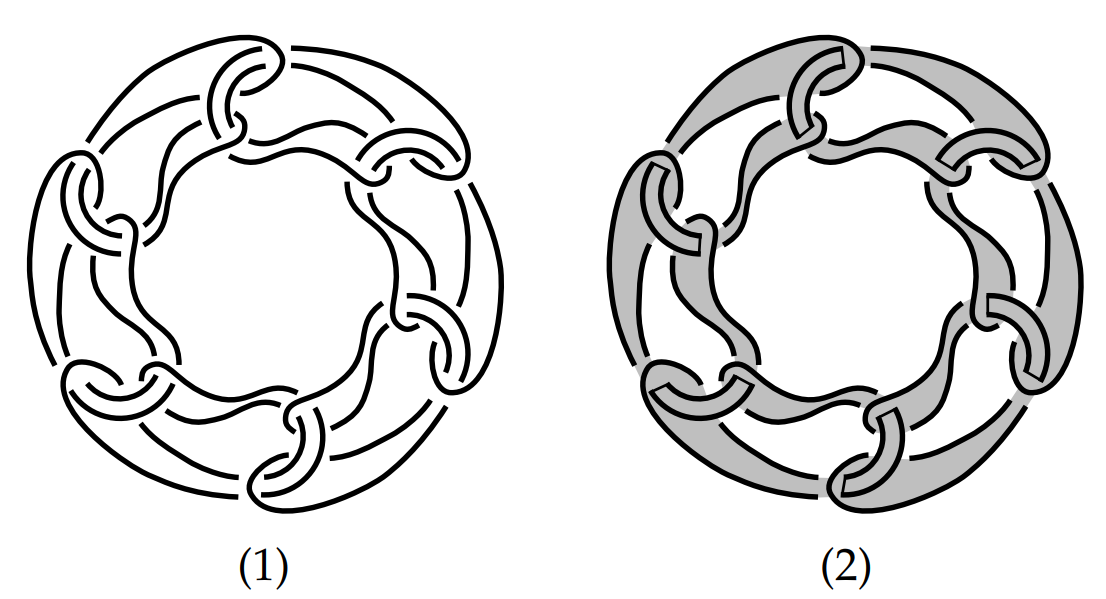}
    \caption{Brunn's chain.}
    \label{fig:brunnlink}
\end{figure}

Consider Brunn's chain in Fig. \ref{fig:brunnlink}(1). Take the disks as shown in Fig. \ref{fig:brunnlink}(2), which are all credible as demonstrated in \cite{BW}. The union of these disks forms a disk system, say $U$. Since $S^3 - int N(U)$ is a handlebody, where $N(U)$ is a regular neighborhood of $U$, by Theorem \ref{thm:untie}, Brunn's chains are untied.
\end{example}

From this example, we can see that this method is intuitive and efficient. Our method works fast to detect untiedness for all Brunnian links before the appearance of \cite{BW}. To our best knowledge, all the Brunnian links in literature before the s-tie operation was proposed by \cite{BM} are untied.

For the very complicated Brunnian links constructed in \cite{BW}, Snakes, Wheels and Cirrus are untied. We show Cirrus (Fig. \ref{fig:cirrus}) is untied in Appendix \ref{subsect:appendix2}, and for the other two, the processes are similar but simpler. For the rest two extremely complicated links in \cite{BW}, Fountains and Jade-pendant, our idea is still similar, but as the process is more complicated, we will show that they are untied in a future note.

\subsection{Proof of Theorem \ref{thm:untie}}\label{subsect:proofuntiedness}
\qquad

\textbf{The ``if'' implication:}

Suppose that $L$ is an s-tie. Let $T$ be a knotted essential torus in $S^3 - L$, bounding a solid torus $V$ in $S^3$. By Theorem \ref{thm:tie}, we know $V$ contains $L$. We may assume $T$ is an \emph{outermost} essential torus. More formally, there is no essential torus $T_0$ in $S^3 - L$ so that

\begin{itemize}
  \item[(0i)] $T_0$ bounds a solid torus containing $V$,
  \item[(0ii)] $T_0$ is not parallel to $T$.
\end{itemize}

For any $i \le k$, the intersection $D_i \cap T$ is a union of circles. So $\cup_{i=1}^k D_i \cap T$ forms a graph in $T$. We wish to make $D_i$ into $V$ one by one by an isotopy of $T$. First consider $D_1$. We do this in three steps.

\textbf{Step 1.} Eliminate inessential circles in $T$.

Let $C$ be a circle component of $D_1 \cap T$, bounding a disk $D_{CT}$ in $T$ and a disk $D_{C1}$ in $D_1$. Since $D_1$ is credible, $D_{C1}$ does not intersect $L$. Thus we have the following two observations:

\begin{itemize}
  \item[(i)] Any circle component of $D_1 \cap T$ contained in $D_{C1}$ is inessential in $T$. This is because $T$ is essential.
  \item[(ii)] For any $i > 1$, there is no circle component of $D_1 \cap D_i$ contained in $D_{C1}$. This is by regularity condition (Rii).
\end{itemize}

By (i), we may assume $C$ is innermost in $D_1$ as a circle component of $D_1 \cap T$. Then $D_{C1} \cup D_{CT}$ bounds a 3-ball, say $B$. Pushing $D_{CT}$ across $B$ a little beyond $D_{C1}$, we eliminate $C$. Do it inductively. After finitely many steps, $D_1 \cap T$ contains no inessential circles in $T$.

To guarantee the validity of making $D_i (i>1)$ into $V$ later, we will need the following two more observations:

\begin{itemize}
  \item[(iii)] The singular part of $U$ contained in $D_{C1}$ consists of proper arcs. This follows from that $D_{C1}$ does not intersect $L$ and (ii).
  \item[(iv)] For any $i > 1$, $D_i \cap B$ is a proper surface, and each boundary component of it is either contained in $D_{CT}$ or is an alternating sequence of properly embedded arcs in $D_{CT}$ and in $D_{C1}$. This follows from (iii).
\end{itemize}

See Fig. \ref{fig:push} for an example.

\begin{figure}[htbp]
		  \centering
    \includegraphics[width=0.80\textwidth]{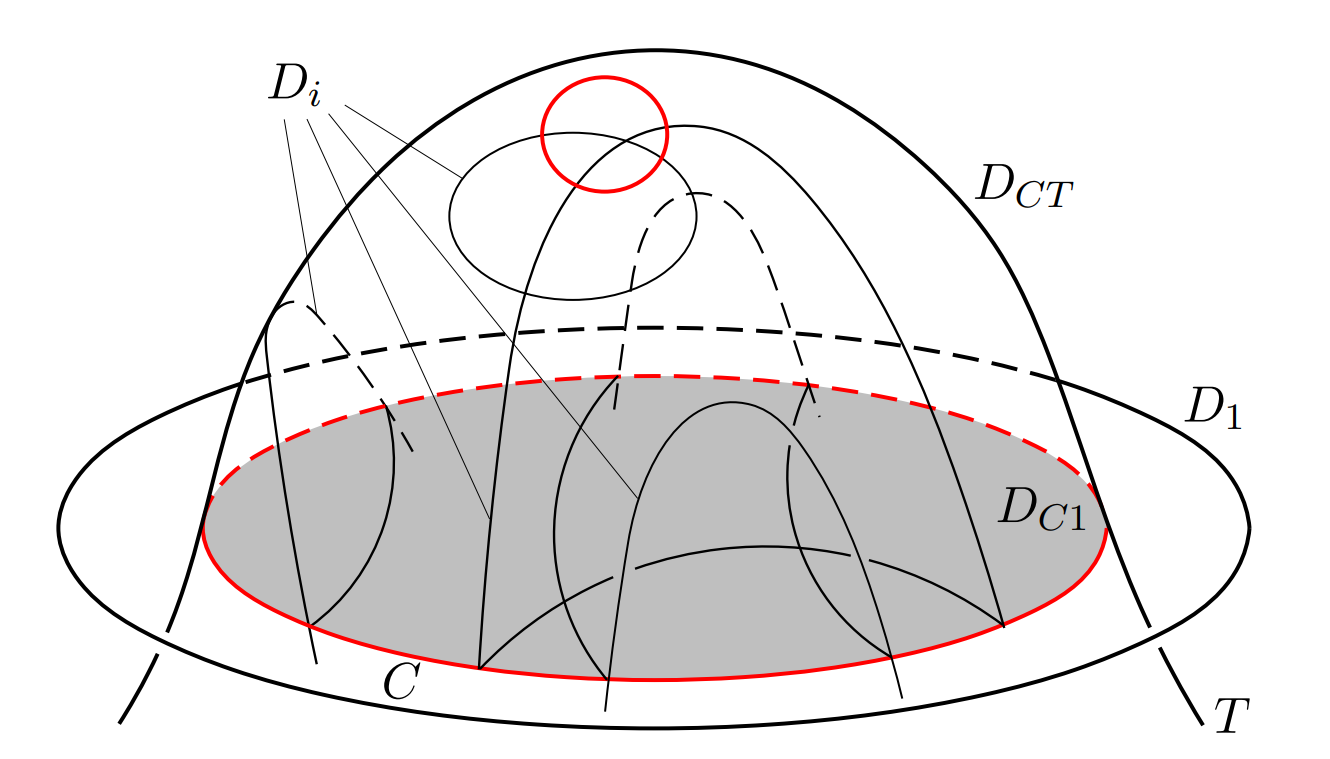}
    \caption{The red circles in $D_{CT}$ are intersections with $D_1$. The black graph is intersection with $D_i (i>1)$.}
	\label{fig:push}
\end{figure}

\textbf{Step 2.} Eliminate exterior annulus regions in $D_1$.

Consider a component of $D_1 \cap T$ innermost in $D_1$. Since it bounds a disk contained in $D_1$, it must be the meridian of $V$. Thus $D_1 \cap T$ is a union of meridians of $V$, cutting $D_1$ into regions alternatively in $intV$ and in $S^3 - V$. We call a component of $D_1 -T$ in $intV$ an \emph{interior region} in $D_1$, and call a component of $D_1 -T$ in $S^3 - V$ an \emph{exterior region} in $D_1$. Clearly every exterior region does not intersect $L$.

We claim that if annulus $A$ is an exterior region in $D_1$, then $A$ is $\partial$-parallel to $T$. In fact, if otherwise, then $\partial A$ would cut $T$ into two annuli $A_1$ and $A_2$ such that both $A \cup A_1$ and $A \cup A_2$ bound knot complements in $S^3 - intV$. This implies that the core of $V$ is a connected sum of two knots, contradicting that $T$ is outermost.

So we can isotope $T$ to make $A$ into $V$. Step by step, we can eliminate all the exterior annulus regions in $D_1$.

\textbf{Step 3.} Verify $D_1 \subset V$.

Consider $D_1$ cut by $D_1 \cap T$. We have the following three observations on the regions in $D_1$:

\begin{itemize}
  \item[(v)] The outermost region and all of the innermost regions are interior, in view of $C_1 \subset V$ and that every intersection circle is a meridian of $V$.
  \item[(vi)] If the outermost region is an annulus, it must intersect $L -C_1$. In fact, if not, $C_1$ would be parallel to a meridian of $T$. However, by Brunnian property and Theorem 1.0.5(1) in \cite{BM}, $L - C_1$ is trivial in $V$. Then $L$ would also be trivial in $V$, a contradiction.
  \item[(vii)] For any exterior region $\Omega$, $\partial \Omega$ should be one outer circle enclosing odd number($\ge 3$) of inner circles in $D_1$, and the disk bounded by each inner circle intersects $L$ in at least $2$ points. In fact, since $\partial \Omega$ is null-homological in $S^3 - intV$, $\partial \Omega$ has an even number of components, and by Step 2, it can not have only 2 components. If a disk bounded by an inner circle intersects $L$ in only one point, a component of $ L $ would be geometrically essential in $ V $, contradicting Theorem 1.0.5(1) in \cite{BM}.
\end{itemize}

We now show by cases that $D_1 \subset V$.

{\sc Case~1:} $D_1$ is stable. Suppose $D_1$ has exterior regions. Then the components of $D_1 \cap T$ can not be all innermost in $D_1$. By (v), each innermost region in $D_1$ is a meridian disk of $V$. There exist such a disk region, say $D_{u1}$, and a component of $D_1 \cap T$ adjacent to $\partial D_{u1}$ in $T$, say $C_{u+1, 1}$, such that

\begin{itemize}
  \item[(1i)] $\sharp (D_{u1} \cap L )$ is minimal among all the innermost regions,
  \item[(1ii)] $C_{u+1, 1}$ is either not innermost in $D_1$, or innermost but not achieving minimal intersection points with $L$.
\end{itemize}

Let $A_T$ be the annulus in $T$ between $C_{u+1, 1}$ and $\partial D_{u1}$. Then by (vii), the disk $D_{u1} \cup A_T$ has less intersection points with $L$ than the disk contained in $D_1$ and bounded by $C_{u+1, 1}$. So $C_{u+1, 1}$ is an unstable circle on $D_1$, a contradiction.

{\sc Case~2:} $\sharp (D_1 \cap L )<7$, if $n=2$ and $lk(C_1, C_2)= \pm 1$; $\sharp (D_1 \cap L )<8$, in other cases. Suppose $D_1$ has exterior regions. By (vi) and (vii), $D_1$ intersects $L - C_1$ in more than 6 points. If the outermost region intersects another component, say $C_2$, in only one point, then since the linking number of $C_2$ and any meridian circle of $ V $ is $ 0 $ by Theorem 1.0.5(1) in \cite{BM}, we have $lk(C_1, C_2)= \pm 1$ and thus $n=2$. As shown in Fig \ref{fig:78points}(1), we have $\sharp (D_1 \cap L ) \geq 7$. Otherwise, $\sharp (D_1 \cap L ) \geq 8$, as shown in Fig. \ref{fig:78points}(2). So under this condition, $D_1 \subset V$.

\begin{figure}[htbp]
		  \centering
    \includegraphics[width=0.80\textwidth]{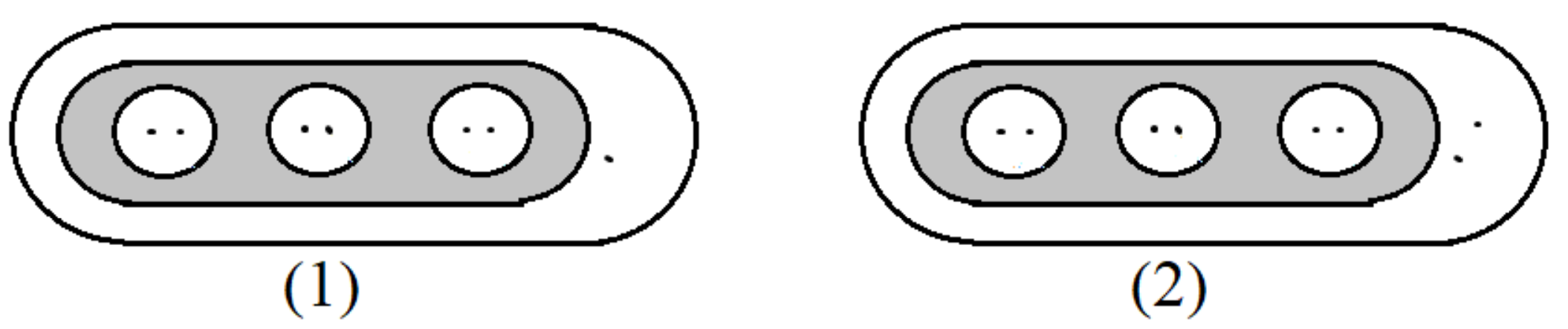}
    \caption{}
	\label{fig:78points}
\end{figure}

Now we follow the three steps above to make $D_2$ into $V$ by an isotopy of $T$. The key point is that this isotopy of $T$ will keep $D_1 \cap T = \emptyset$.

Consider $D_2 \cap T$. For Step 1, by the same token, choose a circle $C$ innermost in $D_2$, which bounds a disk $D_{CT}$ in $T$ and a disk $D_{C2}$ in $D_2$, and then push $D_{CT}$ across the 3-ball cobounded with $D_{C2}$ to eliminate $C$. We have two similar observations for $D_2$ as clauses (iii) and (iv) for $D_1$. Notice that $D_{CT}$ and $D_1$ do not intersect. Thus $D_{C2}$ also does not intersect $D_1$. Therefore the isotopy of $D_{CT}$, replacing $D_{CT}$ by $D_{C2}$, keeps $D_1 \cap T = \emptyset$. So we can eliminate all the inessential circles in $T$ in finitely many steps. For Step 2, we can eliminate exterior annulus regions in $D_2$ just as before because $D_1$ is already contained in $V$. For Step 3, the same argument as demonstrated before shows that $D_2 \subset V$.

Performing the same procedure for $D_i$ for $i$ from $2$ to $k$, we have $U \subset V$. Since $T$ is incompressible in $S^3 - L$, clearly $T$ is incompressible in $S^3 - U$.

\textbf{The ``only if'' implication:}

Suppose $T$ is an incompressible torus in the complement of $U$. Let $V$ be the solid torus bounded by $T$ containing $U$. We prove by negation that $L$ is an s-tie. If $L$ is untied, then $T$ is compressible in the complement of $L$ and thus there is a meridian disk $D$ of $V$ not intersecting $L$. Notice that any meridian disk of $V$ intersects $U$. It follows that for any $i \in I$, $D \cap D_i$ is a disjoint union of circles. The intersection $\bigcup_{i \in I} D_i \cap D$ forms a graph.

Let $C$ be a component of $D_1 \cap D$ innermost in $D_1$, bounding $D_{C1}$ in $D_1$ and bounding $D_C$ in $D$. Since $D_1$ is credible, $D_{C1}$ does not intersect $L$. By (Rii), the singular part of $U$ contained in $D_{C1}$ consists of proper arcs. Replacing $D_C$ in $D$ by $D_{C1}$ and pushing it a little beyond $D_{C1}$, we get a new disk $D$ with less intersection components with $D_1$. After finitely many steps, $D$ is disjoint from $D_1$.

A careful examination shows that the same procedure for $D_2$ keeps $D_1 \cap D = \emptyset$. Step by step for each $D_i$, we eventually get a meridian disk of $V$ not intersecting $U$, a contradiction.

\section{Simple intersection pattern theorem}\label{sect:simpleintersection}

To show a Brunnian link is s-prime is to prove no essential torus splits it into two sublinks. We divide our strategy in two steps. In this section, as the first step, we give a general intermediate result, stating that a disk system satisfying certain conditions can intersect a splitting torus in a simple form. In the next section, as the second step, we give subcriteria to negate the existence of such torus.

\subsection{Simple intersection pattern}\label{subsect:simpleintersectionpattern}

We first define several disks with respect to a partition of a link.

\begin{defn}\label{defn:interior and cross disks}
Let $L=\cup_{i=1}^n C_i$ be an $n$-component link. Let $L_I= \cup_{i \in I} C_i$ and $L_J= \cup_{j \in J} C_i$ be proper sublinks of $L$ with $I \sqcup J = \{1,...,n\}$. For any $i \in I$ and $j \in J$, the credible disks defined as follows are called \emph{$L_I$-interior disk}, \emph{$L_J$-interior disk}, \emph{exterior cross disk} and \emph{free cross disk} respectively, denoted $D_i ^I$, $D_j ^J$, $D_i ^E$ and $D_i ^F$ respectively.

\begin{itemize}
  \item[$D_i ^I$]:  bounded by $C_i$ such that $D_i ^I \cap L_J = \emptyset$;
  \item[$D_j ^J$]:  bounded by $C_j$ such that $D_j ^J \cap L_I = \emptyset$;
  \item[$D_i ^E$]:  bounded by $C_i$ such that $D_i ^E \cap (L_I-C_i) = \emptyset$ and $D_i ^E \cap L_J \neq \emptyset$;
  \item[$D_i ^F$]:  bounded by $C_i$ such that $D_i ^F \cap (L_I-C_i) \neq \emptyset$
and $D_i ^F \cap L_J \neq \emptyset$.
\end{itemize}

$L_I$-interior disk and $L_J$-interior disk are called \emph{interior disks}, and exterior cross disk and free cross disk are called \emph{cross disks}.
\end{defn}

Suppose $T$ is a torus in the complement of $L$ splitting $S^3$ into two solid tori $V_I$ and $V_J$, such that $L \cap V_I = L_I$ and $L \cap V_J = L_J$. Consider the disks defined above and their intersections with $T$.

\begin{defn}\label{defn:simple intersection pattern}
An interior disk intersects $T$ in \emph{simple intersection pattern} if it does not intersect $T$, and a cross disk intersects $T$ in \emph{simple intersection pattern} if each intersection circle is innermost in the disk and is a meridian of $V_J$.
\end{defn}

\begin{figure}[htbp]
		  \centering
    \includegraphics[width=\textwidth]{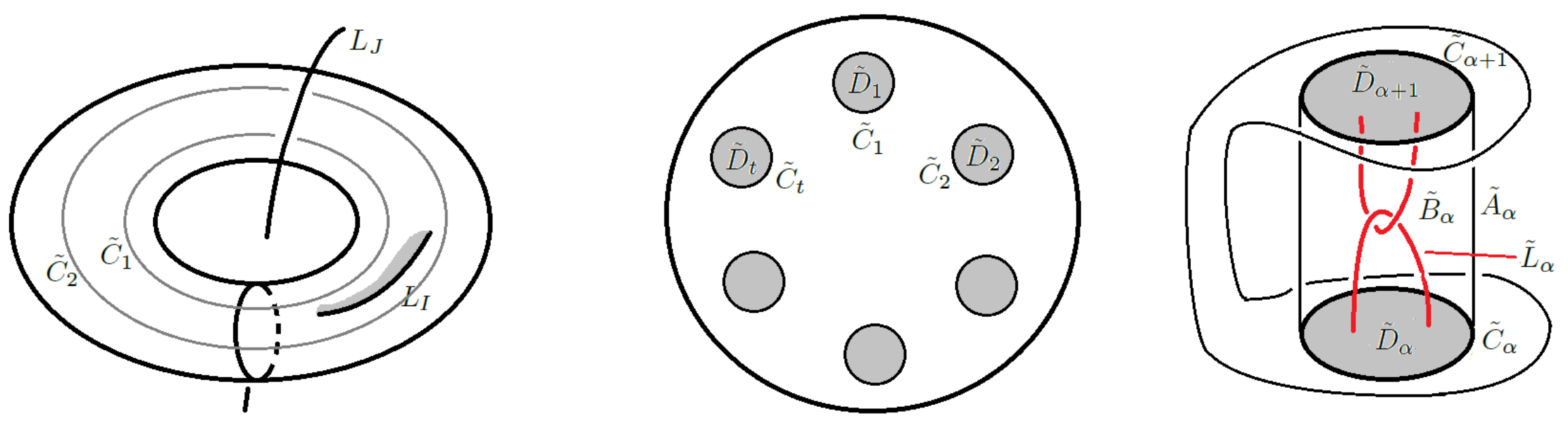}
    \caption{Simple intersection pattern for cross disk in Def. \ref{defn:simple intersection pattern}.}
    \label{fig:simpleintersectionpattern}
\end{figure}

\begin{rem}\label{rem:crossdisk}
Let $\beta _I$ and $\beta _J$ be the cores of $V_I$ and $V_J$ respectively. Consider a cross disk $D_i$ intersecting $T$ in simple intersection pattern, such as Fig. \ref{fig:simpleintersectionpattern}(2).

(1) Recall that cross disks are credible. It is easy to see that each inner disk contained in it (the gray subdisks in Fig. \ref{fig:simpleintersectionpattern}(2)) is a credible disk for $\beta _I \cup L_J$, and the outer region in it (the white region in Fig. \ref{fig:simpleintersectionpattern}(2)) corresponds to a credible disk for $L_I \cup \beta _J$.

(2) If cross disk $D_i$ is stable, consider $\beta _I \cup L_J$. It can be verified that each inner disk contained in $D_i$ is a stable disk bounded by $\beta _I$, and thus all the inner disks have the same number of intersection points with $L_J$.
\end{rem}

Now we turn to Brunnian links.

\begin{prop}\label{prop:simplepattern}
Let $L=\cup_{i=1}^n C_i$ be an $n$-component Brunnian link, split by a torus into two proper sublinks $L_I$ and $L_J$.

\begin{itemize}
  \item[(1)] For any exterior cross disk $D_i ^E$, we have $\sharp (D_i ^E \cap (L - C_i)) \ge 4$.
  \item[(2)] For any free cross disk $D_i ^F$, we have $\sharp (D_i ^F \cap (L - C_i)) \ge 6$.
\end{itemize}
\end{prop}

\begin{rem}\label{rem:equal}
If $\sharp (D_i ^E \cap (L - C_i)) = 4$, then $D_i ^E$ is stable for $L_J$.
\end{rem}

The proofs of this proposition and the following main theorem will be given at the end of this section.

\begin{thm}\label{thm:simplepattern}
(Simple Intersection Pattern Theorem)

Let $L=\cup_{i=1}^n C_i$ be an $n$-component Brunnian link. Let $T$ be an essential torus in the complement of $L$, splitting $S^3$ into two solid tori $V_I$ and $V_J$ with $L \cap V_I = L_I$ and $L \cap V_J = L_J$, where $L_I= \cup_{i \in I} C_i$ and $L_J= \cup_{j \in J} C_j$. 

Let $Q, R, S$ be mutually disjoint subsets of $I$ and $J_0$ be a subset of $J$, and let $D_q ^I$($q \in Q$), $D_j ^J$($j \in J_0$), $D_r ^E$($r \in R$) and $D_s ^F$($s \in S$) be $L_I$-interior disks, $L_J$-interior disks, exterior cross disks, and free cross disks respectively.

Assume that the cross disks are mutually disjoint, assume that
$$\left\{
  \begin{array}{ll}
    D_q ^I \ is \ s \vee 8, & \forall q \in Q; \\
    D_j ^J \ is \ s \vee 8, & \forall j \in J_0; \\
    D_r ^E \ is \ s \vee 10, & \forall r \in R; \\
    D_s ^F \ is \ stable, \ or \ contains \ a \ longitude \ of \ V_I \ with \  \sharp (D_s ^F \cap (L - C_s)) < 8, & \forall s \in S,
  \end{array}
\right.$$
and assume that the union of these disks and $L$ forms a disk system. 

Then after an isotopy of $T$ in $ S^3 - L $, every disk intersects $T$ in simple intersection pattern.
\end{thm}

\begin{rem}\label{rem:conditions}
We explain the conditions in this theorem in more detail.

(1) The cross disks $D_r ^E$($r \in R$), $D_s ^F$($s \in S$) are required to be mutually disjoint. As will be seen in the proof, $\cup_{q \in Q} D_q ^I$ and $\cup_{j \in J_0} D_j ^J$ are actually disjoint.

(2) In the last line in the curly bracket, the condition ``$D_s ^F$ contains a longitude of $V_I$'' seems not natural. But actually once $R \ne \emptyset$, $D_s ^F$ has to contain a longitude of $V_I$, since the cross disks are mutually disjoint.
\end{rem}

\subsection{Proof of Theorem \ref{thm:simplepattern}}\label{subsect:proofsimpleintersection}

We first present a simple lemma on s-sum decomposition of Brunnian links.

\begin{lem}\label{lem: hopf}
\cite{BM} Let $T$ be an essential torus in the complement of a Brunnian link $L$, splitting $L$ into proper sublinks $L_1 \subset V_1$ and $L_2 \subset V_2$, where $V_i$ is a solid torus bounded by $T$, for $i=1, 2$.

(1) Any meridian disk in $V_i$ intersects $L_i$ with at least 2 points, for $i=1, 2$;

(2) If a component of $L_1$ bounds a disk whose intersection with $T$ contains a meridian of $V_2$, then the intersection contains at least two meridians of $V_2$.
\end{lem}

\begin{proof}[Proof of Theorem \ref{thm:simplepattern}]
Let $U$ denote the disk system formed by $L$ and all the credible disks. Then $U \cap T$ is a union of circles, forming a graph in $T$. We will isotope $T$ to make the $L_I$-interior disks, the $L_J$-interior disks and the cross disks into simple intersection pattern in turn. So we divide our proof in three steps.

\textbf{Step 1.} Make $L_I$-interior disks into $V_I$ one by one by an isotopy of $T$.

We shall adopt the same procedure as in the proof of Theorem \ref{thm:untie}. Consider a disk $D_q ^I$. We make it into $V_I$ in three steps.

\textbf{Step 1.1.} Eliminate inessential circles in $T$.

Let $C$ be a circle component of $D_q ^I \cap T$, bounding a disk $D_{CT}$ in $T$ and a disk $D_{Cq}$ in $D_q ^I$. Using the same argument as in the proof of Theorem \ref{thm:untie}, we may assume $C$ is innermost in $D_{Cq}$. Then $D_{Cq} \cup D_{CT}$ bounds a 3-ball. Since no sphere splits a Brunnian link, this 3-ball does not intersect $L$. So we can eliminate $C$ by an isotopy of $T$ as demonstrated before. Do it inductively. After finitely many steps, $D_q ^I \cap T$ contains no inessential circle in $T$.

\textbf{Step 1.2.} Eliminate annulus regions in $V_J$.

Now $D_q ^{I} \cap T$ is a disjoint union of circles. Consider a circle component innermost in $D_q ^I$. Since the disk contained in $D_q ^I$ and bounded by this circle does not intersect $L_J$, it is a meridian disk of $V_I$. Thus all the components of $D_q ^{I} \cap T$ are meridians of $V_I$.

The intersection $D_q ^{I} \cap T$ cuts $D_q ^I$ into regions alternatively in $V_I$ and in $V_J$. Let $A$ be an annulus region contained in $V_J$. Since the boundary components of $A$ are longitudes of $V_J$, it follows that $A$ must cut $V_J$ into two solid tori. Notice that the core of each solid tori and the core of $V_I$ form a Hopf link, 
so these three cores form a $ 3 $-component link, which is the connected sum of two Hopf links, not Brunnian. Hence by Theorem \ref{thm:sum}, one of the two solid tori does not intersect $L_J$. Isotope $T$ across this solid torus to make $A$ into $V_I$. Step by step, we can eliminate all the annulus regions in $V_J$.

\textbf{Step 1.3.} Verify $D_q ^I \subset V_I$.

For either the case that $D_q ^I$ is stable or the case $\sharp (D_q ^I \cap L_J ) < 8$, an argument similar to the one used in Step 3 in the proof of Theorem \ref{thm:untie} shows that $D_q ^I \subset V_I$.

Next we isotope $T$ to make another $L_I$-interior disk $D_{q'} ^{I}$ into $V_I$ by the same token. We now check that the isotopy of $T$ keeps $D_q ^I \cap T = \emptyset$. For Step 1.1, similarly as in the proof of Theorem \ref{thm:untie}, this is due to the regularity condition (Rii) and that $D_q ^I$ does not intersect any disk contained in $T$. For Step 1.2, it is guaranteed by that $D_q ^I$ is already contained in $V_I$.

Performing the same procedure for $L_I$-interior disks one by one, we have that $\cup_{q \in Q} D_q ^I \subset V_I$.

\textbf{Step 2.} Make $L_J$-interior disks into $V_J$ one by one by an isotopy of $T$.

With $\cup_{q \in Q} D_q ^I$ already contained in $V_I$, we proceed as in Step 1 to make $L_J$-interior disks into $V_J$. Consider a disk $D_j ^J$. We only need to check that in the similar three steps, the isotopy of $T$ keeps $D_q ^I \cap T = \emptyset$ for all $q \in Q$.

\textbf{Step 2.1.} Eliminate inessential circles in $T$.

In this step, the disk $D_{Cj}$ contained in $D_j ^J$, corresponding to $D_{Cq}$ in Step 1.1, does not intersect any $L_I$-interior disk by (Rii). Thus the 3-ball it caps with $T$ does not intersect any $L_I$-interior disk.

\textbf{Step 2.2.} Eliminate annulus regions in $V_I$.

In such a case, $D_j ^J \cap L_I = \emptyset$ implies that $D_j ^J \cap T$ can only be a disjoint union of
meridians of $V_J$. The intersection circles cut $D_j ^J$ into regions alternatively in $V_I$ and in $V_J$. Let $A$ be an annulus region contained in $V_I$. An argument similar to the one used in Step 1.2 shows that $A$ caps a solid torus with $T$, say $V_A$, which does not intersect $L_I$.

We claim that $A \cap D_q ^I =\emptyset$, and thus $V_A \cap D_q ^I =\emptyset$, for all $q \in Q$. In fact, by (Rii), there is no component of $D_q ^I \cap A$ inessential in $A$ for any $q \in Q$. Suppose $C_A$ is a component of $D_q ^I \cap A$ essential in $A$ for a $q \in Q$. Since $C_A$ is a core of $V_I$, the disk contained in $D_q ^I$ and bounded by $C_A$ must intersect $L_J$, a contradiction.

An isotopy of $T$ across $V_A$ makes $A$ into $V_J$.

\textbf{Step 2.3.} Verify $D_j ^J \subset V_J$.

This step is a judgement, where everything is fixed.

Perform the same procedure for other $L_J$-interior disks one by one. We can verify that $T$ keeps avoiding $\cup_{q \in Q} D_q ^I$ and the $L_J$-interior disks already contained in $V_J$ by the same arguments as in Step 1 and Step 2 respectively. Now we have $\cup_{j \in J_0} D_j ^J \subset V_J$.

\textbf{Step 3.} Make the cross disks into simple intersection pattern.

Recall that all of the cross disks are mutually disjoint, so it is convenient to consider them together.  We proceed as in the proof of Theorem \ref{thm:untie}.

\textbf{Step 3.1.} Eliminate inessential circles in $T$.

The intersection of the cross disks and $T$ consists of mutually disjoint circles. Let $D_i$ be one of the exterior or free cross disks. Let $C$ be a circle component of $D_i \cap T$, bounding a disk $D_{CT}$ in $T$ and a disk $D_{Ci}$ in $D_i$. Using the same argument as in Step 1 in the proof of Theorem \ref{thm:untie}, we may assume $C$ is innermost in $D_{Ci}$. Then $D_{Ci} \cup D_{CT}$ bounds a 3-ball. As demonstrated in Step 1.1, this 3-ball does not intersect $L$, and thus we can eliminate $C$ by an isotopy of $T$ .

Step by step for all cross disks, we eliminate all inessential intersection circles in $T$.

\textbf{Step 3.2.} Eliminate annulus regions $\partial$-parallel to $T$.

It is easy to see that each component of $(\cup_{r \in R} D_r ^E \cup \cup_{s \in S} D_s ^F) \cap T$ is either the longitude of $V_I$ or the meridian of $V_I$. Let $D_i$ be one of the exterior or free cross disks. The intersection $D_i \cap T$ cuts $D_i$ into planar regions alternatively in $V_I$ and in $V_J$. In this step, we eliminate all the proper annulus regions $\partial$-parallel to $T$ without intersecting $L$. Strictly speaking, let $A$ be a proper annulus region contained in $V_I$ which caps a solid torus with $T$, say $V_A$, such that $V_A \cap L =\emptyset$. We claim that $A \cap D_q ^I =\emptyset$, and thus $V_A \cap D_q ^I =\emptyset$, for any $q \in Q$. Then an isotopy of $T$ across $V_A$ makes $A$ into $V_J$.

We now prove the claim. By (Rii), there is no component of $D_q ^I \cap A$ inessential in $A$ for any $q \in Q$. Suppose $C_A$ is a component of $D_q ^I \cap A$ essential in $A$ for a $q \in Q$. Let $P$ be a component of $D_q ^I \cap V_A$ one of whose boundary components is $C_A$. Clearly the boundary components of $P$ are all parallel in $A$. By (Riii), $P$ is not an annulus. Then $P$ is compressible in $V_A$. Let $D_{P1}$ be a compression disk for $P$ in $V_A$. Since $D_q ^I$ is credible, the disk contained in $D_q ^I$ and bounded by $\partial D_{P1}$, denoted $D_{P0}$, does not intersect $L$. Without loss of generality, we may reselect $C_A$ so that it is contained in $D_{P0}$. However, since $C_A$ is essential in $A$, it is parallel to either the longitude or the meridian of $V_I$. In either case, any disk bounded by $C_A$ intersects $L$, a contradiction.

We can eliminate such annulus regions contained in $V_J$ in a similar manner.

\textbf{Step 3.3.} Verify simple intersection pattern.

Consider all of the cross disks cut by $T$. For both exterior and free cross disks, the following claim investigates the remaining annulus regions. To make length short, we arrange the proof of this claim after the end of the proof of the theorem.
\textbf{Claim:} After Step 3.2,
\begin{itemize}
  \item [(Ai)] any proper annulus region in $V_I$ intersects $L_I$, and intersects any component of $L_I$ in an even number of points;
  \item [(Aii)] any annulus region in $V_J$ intersects $L_J$, and intersects any component of $L_J$ in an even number of points.
\end{itemize}

In the remainder of the proof, we discuss exterior cross disks and free cross disks separately.

\textbf{Exterior cross disks.} For any $r \in R$, the intersection $D_r ^E \cap T$ cuts $D_r ^E$ into regions alternatively in $V_I$ and $V_J$. Since $D_r ^E$ is an exterior cross disk, the intersection circles are longitudes of $V_I$. We have the following observations on the regions in $D_r ^E$.

\begin{itemize}
  \item[(Ei)]  Each innermost region is a meridian disk of $V_J$. In view of linking number, each innermost disk intersects any componens of $L_J$ in an even number of points unless $L_J$ has only one component.
  \item[(Eii)] The outermost region is contained in $V_I$. It is not an annulus by Lemma \ref{lem: hopf}. Its boundary consists of $C_r$ and an even number of intersection circles unless $L_I = C_r$, in view of linking number.
  \item[(Eiii)] No region in $V_I$ intersects $L- C_r$. For each region $\Omega$ in $V_I$, except the outermost one, $\partial \Omega$ is one outer circle enclosing odd number($\ge 3$) of inner circles in $D_r ^E$. In fact, since $\partial \Omega$ is null-homological in $V_I$, such a region has even number($>2$) of boundary components.
\end{itemize}

Now we show by cases that $D_r ^E$ is of simple intersection pattern.

{\sc Case~1:} $D_r ^E$ is stable. Let $D_{ur}$ be an innermost region in $D_r ^E$ so that $\sharp (D_{ur} \cap L_J )$ is minimal among all the innermost regions in $D_r ^E$. Let $C_{u+1,r}$ be a component of $D_r ^E \cap T$ adjacent to $\partial D_{ur}$ in $T$, and $A_T$ be the annulus in $T$ between $\partial D_{ur}$ and $C_{u+1,r}$. Then the disk $D_{ur} \cup A_T$ has the same intersection points with $L$ as $D_{ur}$. Since $D_r ^E$ is stable, the disk bounded in $D_r ^E$ by $C_{u+1,r}$ must be innermost and have the same number of intersection points as $D_{ur}$. It follows that all the intersection circles are innermost in $D_r ^E$.

{\sc Case~2:} $\sharp (D_r ^E \cap (L - C_r)) < 10$. Suppose there is an intersection circle in $D_r ^E$ which is not innermost. By (Aii), (Ei), (Eii) and (Eiii), $D_r ^E$ intersects $L- C_r$ in at least 10 points, as shown in Fig. \ref{fig:10points}.

\begin{figure}[htbp]
		  \centering
    \includegraphics[width=0.90\textwidth]{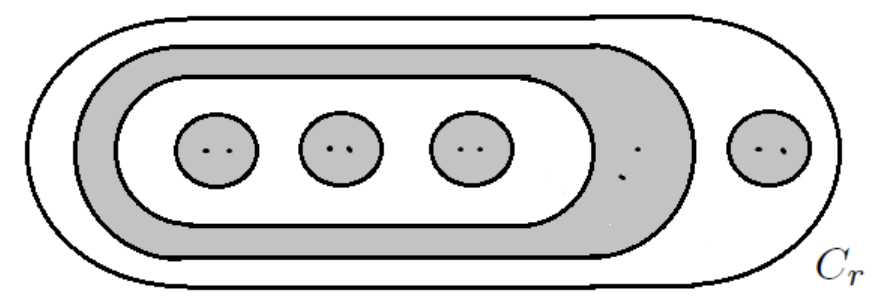}
    \caption{White regions are in $V_I$, and gray regions are in $V_J$.}
    \label{fig:10points}
\end{figure}

\textbf{Free cross disks.} For any $s \in S$, the intersection $D_s ^F \cap T$ cuts $D_s ^F$ into regions alternatively in $V_I$ and $V_J$. The intersection circles can be either longitudes of $V_I$ or
meridians of $V_I$. We discuss the two possibilities separately.

First suppose that the circles are longitudes of $V_I$. We have the following observations on the regions in $D_s ^F$.

\begin{itemize}
  \item[(Fi)] Each innermost region is a meridian disk of $V_J$, and it intersects any component of $L_J$ in an even number of points unless $L_J$ has only one component.
  \item[(Fii)] The outermost region is contained in $V_I$. It is not an annulus by Lemma \ref{lem: hopf}. Its boundary consists of $C_s$ and an even number of intersection circles.
  \item[(Fiii)] For each region $\Omega$ in $V_I$, except the outermost one, $\partial \Omega$ is one outer circle enclosing an odd number of inner circles in $D_s ^F$.
\end{itemize}

Now suppose that the intersection circles are meridians of $V_I$. We have the following observations on the regions in $D_s ^F$.

\begin{itemize}
  \item[(F'i)] The outermost region and all of the innermost regions are contained in $V_I$. Each innermost region is a meridian disk of $V_I$.
  \item[(F'ii)] If the outermost region is an annulus, it intersects $L_I -C_s$, and thus it intersects any component of $L_I -C_s$ in an even number of points, in view of linking number. In fact, if not, $C_s$ would be parallel to a meridian of $V_I$. By Brunnian property, $L_I -C_s$ is trivial in $V_I$. This would imply $L_I$ is trivial in $V_I$, a contradiction.
  \item[(F'iii)] For any region $\Omega$ in $V_J$, $\partial \Omega$ is one outer circle enclosing an odd number of inner circles in $D_s ^F$, and each disk contained in $D_s ^F$ and bounded by one of such inner circles intersects $L_I$.
\end{itemize}

We now show by cases that $D_s ^F$ is of simple intersection pattern.

{\sc Case~1:} $D_s ^F$ is stable. Suppose that the intersection circles are meridians of $V_I$. By (F'i), the components of $D_s ^F \cap T$ can not be all innermost in $D_s ^F$. Notice that each innermost region is a meridian disk of $V_I$. There exists an innermost region in $D_s ^F$, say $D_{us}$, and a component of $D_s ^F \cap T$ adjacent to $\partial D_{us}$ in $T$, say $C_{u+1, s}$, such that
\begin{itemize}
  \item[(1i)] $\sharp (D_{us} \cap L_I )$ is minimal among all the innermost regions,
  \item[(1ii)] $C_{u+1, s}$ is either not innermost in $D_s ^F$, or innermost but not achieving minimal intersection points with $L_I$.
\end{itemize}
Let $A_T$ be the annulus in $T$ between $C_{u+1, s}$ and $\partial D_{us}$. Then by (F'iii) and (Aii), the disk $D_{us} \cup A_T$ has less intersection points with $L$ than the disk contained in $D_s ^F$ and bounded by $C_{u+1, s}$. This implies that $C_{u+1, s}$ is an unstable circle on $D_s ^F$, a contradiction.

Hence $D_s ^F \cap T$ consists of longitudes of $V_I$. The remainder of the argument is analogous to that in Case 1 for exterior cross disks and is left to the reader.

{\sc Case~2:} $\sharp (D_s ^F \cap (L - C_s)) < 8$ and $D_s ^F$ contains a longitude of $V_I$. All of the intersection circles are longitudes of $V_I$. Suppose there is a circle in $D_s ^F$ which is not innermost. By Lemma \ref{lem: hopf}(1), every meridian disk of $V_J$ intersects $L_J$ in at least 2 points. Then by (Fi), (Fii), (Fiii), (Ai) and (Aii), $D_s ^F$ intersects $L- C_s$ in at least 8 points, as shown in Fig. \ref{fig:8pts}.

\begin{figure}[htbp]
		  \centering
    \includegraphics[width=0.70\textwidth]{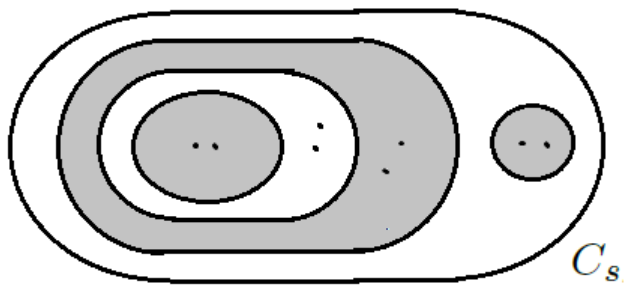}
    \caption{White regions are in $V_I$, and gray regions are in $V_J$.}
    \label{fig:8pts}
\end{figure}

\end{proof}

\begin{proof}[Proof of the claim]
We only prove (Ai), and (Aii) follows in a similar manner by interchanging $V_I$ and $V_J$ in the proof below.

Let $A$ be a proper annulus region in $V_I$. Assume for contrary that $A \cap L_I = \emptyset$. First suppose that all the intersection circles are longitudes of $V_I$. Then $A$ must cut $V_I$ into two solid tori. Since the core of each solid torus and the core of $V_J$ form a Hopf link, by Theorem \ref{thm:sum}, one of the two solid tori does not intersect $L_I$. However, we have eliminated such $A$ in Step 3.2, a contradiction.

Now suppose that all the intersection circles are meridians of $V_I$. If $A$ is $\partial$-parallel in $V_I$, then it cuts off a compression solid torus $V_A$ in $V_I$. If $V_A$ contains some components of $L_I$, by Proposition 3.2.1 in \cite{BM}, the proper sublink in the closure of $V_I - V_A$ would not be geometrically essential. A compression disk in this solid torus, after isotopy, would be a meridian disk in $V_I$. However, again by Proposition 3.2.1 in \cite{BM}, $L_I$ is geometrically essential in $V_I$, a contradiction. So $V_A$ does not intersect $L$. Recall that we have eliminated such $A$ in Step 3.2. Hence $A$ cannot be $\partial$-parallel in $V_I$. Then $\partial A$ cuts $T$ into two annuli such that $A$ and one of the annuli cobound a knotted solid torus $V_{AI}$ in $V_I$. By Theorem \ref{thm:tie}, the sublink in $V_{AI}$ is not geometrically essential. A compression disk in $V_{AI}$, after isotopy, would be a meridian disk in $V_I$, again a contradiction.

We have shown that $A$ intersects $L_I$. Since $A$ separates $V_I$, it intersects any component of $L_I$ in an even number of points.
\end{proof}

\begin{proof}[Proof of Proposition \ref{prop:simplepattern}]
Let $T$ be an essential torus in the complement of $L$, splitting $S^3$ into two solid tori $V_I$ and $V_J$ so that $L \cap V_I = L_I$ and $L \cap V_J = L_J$.

(1) Applying Theorem \ref{thm:simplepattern}, we may assume $D_i^E$ intersects $T$ in simple intersection pattern. By Lemma \ref{lem: hopf}, any meridian disk of $V_I$ has more than one intersection point, and the outermost region has more than two boundary components. So if $\sharp (D_i ^E \cap (L - C_i)) \le 4$ , the only possibility is as shown in Fig. \ref{fig:64pts}(1).

(2) First suppose all the intersection circles in $D_i^F$ are longitudes of $V_I$. Applying Theorem \ref{thm:simplepattern}, we may assume $D_i^F$ is of simple intersection pattern. By Lemma \ref{lem: hopf}(1) and (Fii) in the proof of Theorem \ref{thm:simplepattern}, there are at least two inner disks each intersecting $L_J$ in at least 2 points. In view of linking number, the outer region intersects $L_I -C_i$ in an even number of points. So $D_i^F$ with least number of intersection is as shown in Fig. \ref{fig:64pts}(2).

Now suppose that all the intersection circles in $D_i^F$ are meridians of $V_I$. By Lemma \ref{lem: hopf}(1), and (Aii) and (F'ii) in the proof of Theorem \ref{thm:simplepattern}, if $\sharp (D_i ^F \cap (L - C_i)) \le 6$, the intersections in $D_i^F$ can only be as shown in Fig. \ref{fig:64pts}(3).

\begin{figure}[htbp]
		  \centering
    \includegraphics[width=\textwidth]{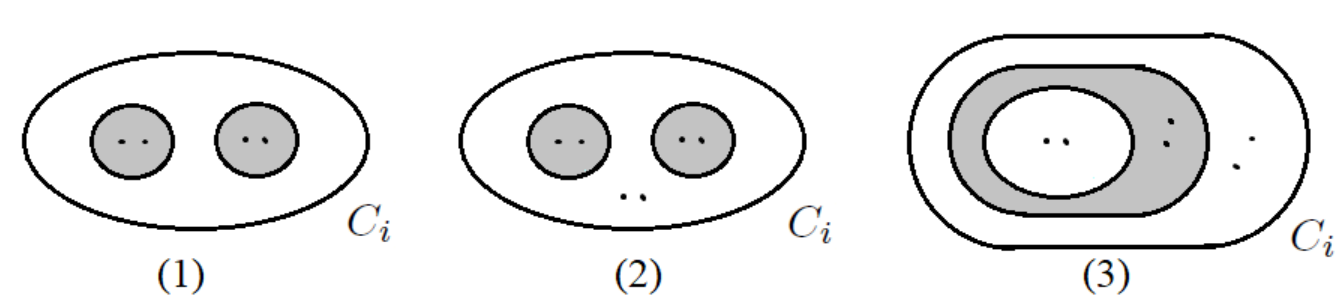}
    \caption{White regions are in $V_I$, and gray regions are in $V_J$.}
    \label{fig:64pts}
\end{figure}

\end{proof}

\section{Subcriteria for S-primeness}\label{sect:sprime}

Based on Theorem \ref{thm:simplepattern} in the previous section, we present four methods to prove a Brunnian link is s-prime and illustrate them by typical examples. 

Using our methods, we can discriminate s-primeness for all the Brunnian links in literature. Especially, the proof for Carpet($m,n,p$) will apply verbatim to show all the new Brunnian links in \cite{BW}, Snakes, Fountains, Cirrus, Jade-pendants and Wheels, are s-prime. Meanwhile, there are many Brunnian links that are not s-prime, such as Milnor links (see \cite{BM}), many links in \cite{BFJRVZ,BCS} and so on.

In this section, we call the following paragraph \emph{Standard premise}:

Let $L=\cup_{i=1}^n C_i$ be an $n$-component Brunnian link. Let $T$ be an essential torus in the complement of $L$, splitting $S^3$ into two solid tori $V_I$ and $V_J$ with $L \cap V_I = L_I = \cup_{i \in I} C_i$ and $L \cap V_J = L_J = \cup_{j \in J} C_i$. Let $Q, R, S$ be mutually disjoint subsets of $I$ and $J_0$ be a subset of $J$, and let $D_q ^I$($q \in Q$), $D_j ^J$($j \in J_0$), $D_r ^E$($r \in R$) and $D_s ^F$($s \in S$) be $L_I$-interior disks, $L_J$-interior disks, exterior cross disks, and free cross disks respectively. Assume that the cross disks are mutually disjoint, the union of these disks and $L$ forms a disk system $U$, and every disk intersects $T$ in simple intersection pattern.

\subsection{Discarding components}\label{subsect:componentsdiscard}

This simple method is powerful when the number of components of the link is large.

\begin{subcriterion}\label{prop: throw}
Under Standard premise, let $U ^\circ$ be $U$ with some connected components in $V_I$ discarded. Assume $R \cup S \ne \emptyset$. Then no connected component of $U ^\circ$ contained in $V_J$ is split out by a sphere.
\end{subcriterion}

\begin{proof}
By Theorem \ref{thm:sum}, the link $L_J \cup \beta _I$ is Brunnian. Thus no sphere splits $L_J \cup \beta _I$. By construction of $U$, no sphere splits $\beta _I \cup (U \cap V_J)$. Notice that a cross disk in simple intersection pattern contains a longitude of $V_I$, which is parallel to $\beta _I$. It follows from the construction of $U$ that in the complex $U$ with some connected components in $V_I$ discarded, no connected component in $V_J$ can be split out by a sphere.
\end{proof}

This proposition provides a method to rule out essential torus splitting $L_I$ and $L_J$, as the hypothetical torus $T$ cannot exist if discarding some components of $U$ in $V_I$ makes some component of $U$ in $V_J$ split out.

\begin{example}\label{example:debrunner}

\begin{figure}[htbp]
		  \centering
    \includegraphics[width=0.90\textwidth]{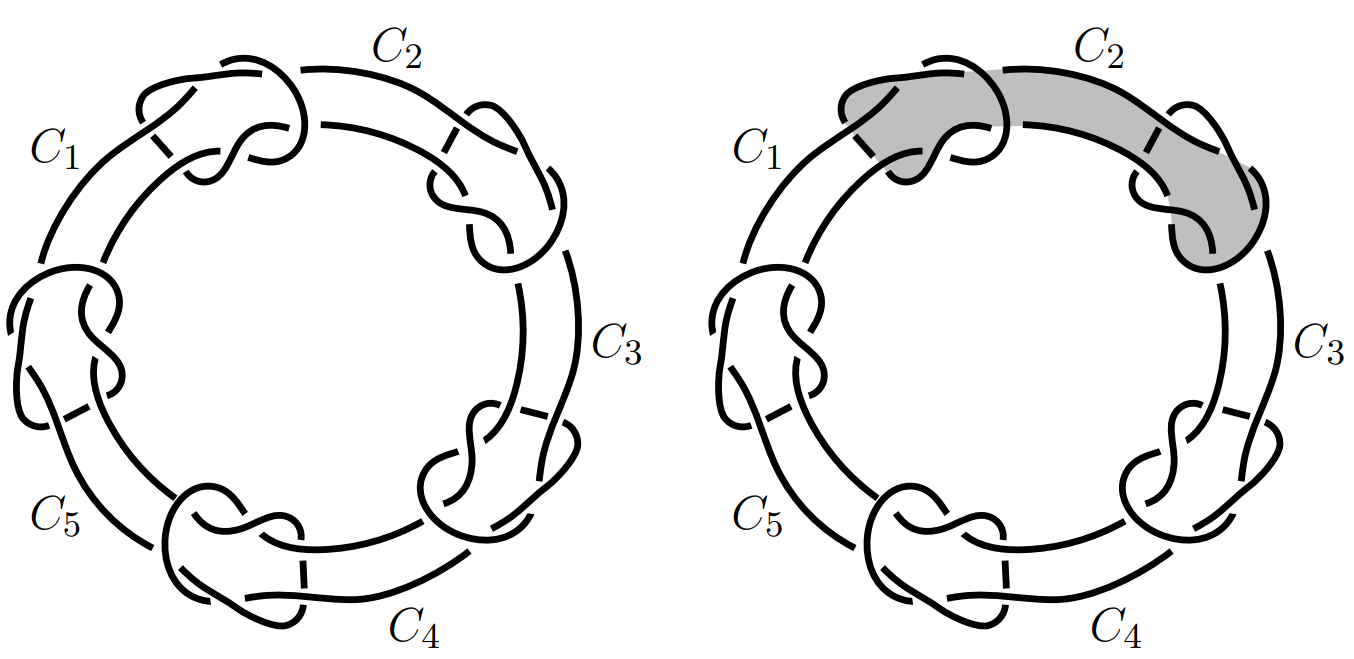}
    \caption{deBrunner(5).}
    \label{fig:debrunner5}
\end{figure}

deBrunner(5) is s-prime. Denote the link in Fig. \ref{fig:debrunner5} by $L$. In view of the symmetry of $L$, it suffices to show there is no essential torus splitting $C_1$ and $C_2$. Assume for contrary that $T$ is such a torus. Take the cross disk $D_2$ bounded by $C_2$ as shown in Fig. \ref{fig:debrunner5}. Depending on whether $T$ splits $C_2$ and $C_3$, the following discussion is complete.
\begin{itemize}
  \item [(i)] $V_I$ contains only $C_2$;
  \item [(ii)] $V_I$ contains exactly $C_2 \cup C_3$;
  \item [(iii)] $V_I$ contains more than $C_2 \cup C_3$ and $V_J$ contains more than $C_1$;
  \item [(iv)] $V_J$ contains only $C_1$;
  \item [(v)] $V_J$ contains exactly $C_1 \cup C_3$;
  \item [(vi)] $V_J$ contains more than $C_1 \cup C_3$ and $V_I$ contains more than $C_2$.
\end{itemize}
None above is possible. The main point of our argument is that in $L \cup D_2$, after deleting a connected component other than $C_1 \cup D_2 \cup C_3$, then all the other components split out.

For (ii) and (iv), they are impossible by Proposition \ref{prop:simplepattern}(2). Clause (i) is equivalent to (iv) by the symmetry of $L$, thus also impossible. For (iii) and (vi), they are impossible by Criterion \ref{prop: throw}. In fact, in $L \cup D$, once we delete a circle component in $V_I$, each circle component in $V_J$ splits out. Finally, (v) is contained in (iii) after a rotation of $L$.
\end{example}

Proposition 6.2.1 in \cite{BM}, the uniqueness of s-sum decomposition for Brunnian links, can simplify the task of showing s-primeness. Specifically, as any two tori splitting $L$ are disjoint after isotopy, if a torus $T$ splits $L$ into $L_I$ and $L_J$ and a torus $T'$ splits $L$ into $L_{I'}$ and $L_{J'}$, then one of $L_I \cap L_{I'}$, $L_I \cap L_{J'}$, $L_J \cap L_{I'}$ and $L_J \cap L_{J'}$ is empty.

\begin{example}\label{example:torus(m,n)}
\begin{figure}[htbp]
		  \centering
    \includegraphics[width=0.50\textwidth]{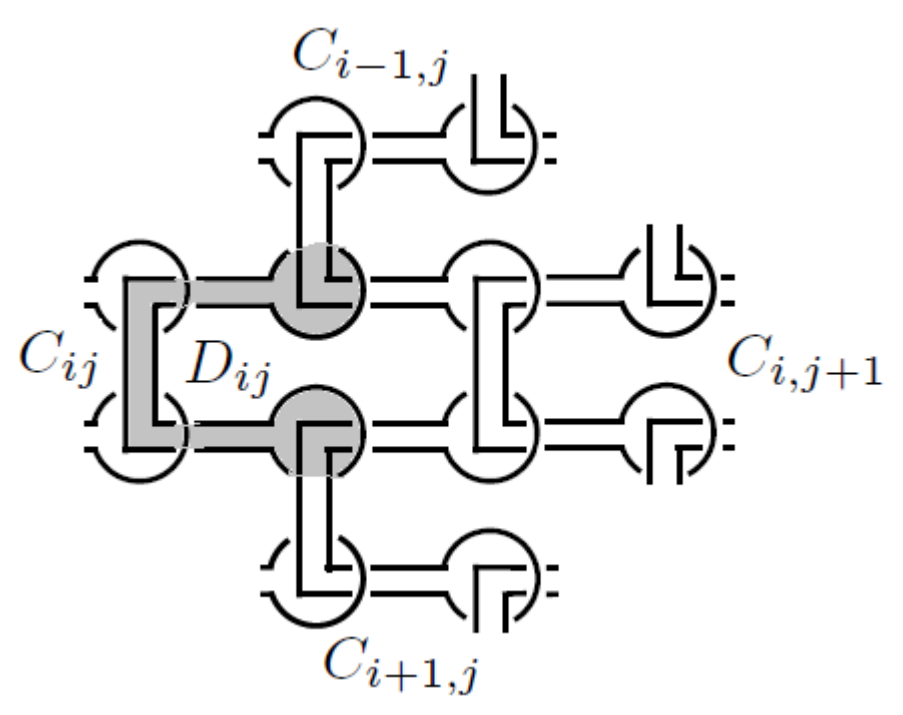}
    \caption{Torus($m,n$) is s-prime.}
	\label{fig:torusprime}
\end{figure}
For any $m>1,n>2$, Torus($m,n$) is s-prime. Up to symmetry, there is only one kind of component, as $C_{ij}$ in Fig. \ref{fig:torusprime}. It suffices to show that no essential torus splits any one of $C_{ij}$, $C_{i-1,j}$ and $C_{i+1,j}$ from the other two. Take the cross disk $D_{ij}$. By Proposition \ref{prop:simplepattern}(2), no essential torus splits $C_{i+1,j}$ and $C_{ij} \cup C_{i-1,j}$ or splits $C_{i-1,j}$ and $C_{ij} \cup C_{i+1,j}$.

Suppose there is an essential torus $T$ splitting $C_{ij}$ and $C_{i-1,j} \cup C_{i+1,j}$. In $L \cup D_{ij}$, once we delete a circle component other than $C_{i,j+1}$, all other circle components split out. Thus by Criterion \ref{prop: throw},  $T$ either splits $C_{i-1,j} \cup C_{i+1,j}$ from all the other components, or splits $C_{ij} \cup C_{i,j+1}$ from all the other components. For the first case, by the symmetry of $L$, there would be a torus splitting $C_{i-3,j} \cup C_{i-1,j}$ from all the other components. by Proposition 6.2.1 in \cite{BM}, this is impossible. Similarly, the second case is impossible.
\end{example}

\begin{defn}
Let $C$ be a component of a Brunnian link $L$. If there is an essential torus in $S^3 - L$ splitting $C$ and $L-C$, then $C$ is \emph{multiple}; otherwise, \emph{simple}.
\end{defn}

\begin{figure}[htbp]
		  \centering
    \includegraphics[width=\textwidth]{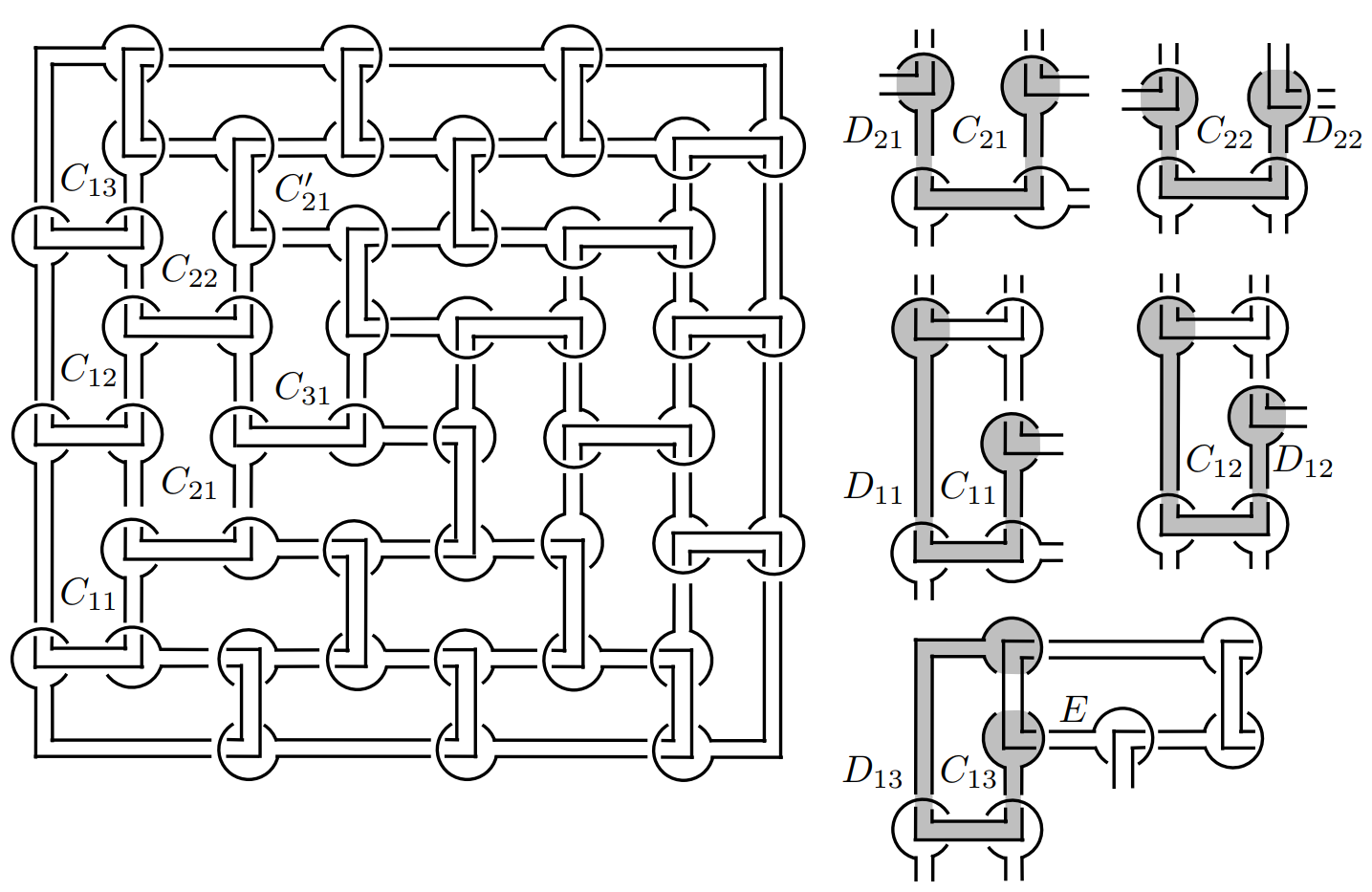}
    \caption{Carpet(1,3,4) is s-prime.}
	\label{fig:carpet}
\end{figure}

\begin{example}\label{example:carpet}
Carpet(1,3,4) is s-prime. Denote this link by $L$. Up to symmetry, there are 6 kinds of components. We will take cross disks in Fig. \ref{fig:carpet} one by one to show $L$ is s-prime.

On account of the symmetry of this link, the following discussion is complete.
\begin{itemize}
  \item [(i)] No essential torus splits one of $C_{11}$, $C_{12}$ and $C_{21}$ from the other two;
  \item [(ii)] No essential torus splits one of $C_{12}$, $C_{13}$ and $C_{22}$ from the other two;
  \item [(iii)] No essential torus splits $C_{31}$ and $C_{12} \cup C_{21}$;
  \item [(iv)] No essential torus splits $C'_{21}$ and $C_{13} \cup C_{22}$;
  \item [(v)] Every component is simple.
\end{itemize}

For (iii), take cross disk $D_{21}$ and use Proposition \ref{prop:simplepattern}(2). For (i), take cross disk $D_{11}$. By Proposition \ref{prop:simplepattern}(2), any essential torus neither splits $C_{12}$ and $C_{11} \cup C_{21}$, nor splits $C_{21}$ and $C_{11} \cup C_{12}$. Suppose there is an essential torus $T$ splitting $C_{11}$ and $C_{12} \cup C_{21}$. In $L \cup D_{11}$, once we delete a circle component, all other circle components split out. By Criterion \ref{prop: throw}, either $C_{11}$ is multiple or $T$ splits $C_{12} \cup C_{21}$ from all the other components.  By (iii), it only remains to show $C_{11}$ is simple.

For (iv), take cross disk $D_{22}$ and use Proposition \ref{prop:simplepattern}(2). For (ii), take cross disk $D_{12}$. By Proposition \ref{prop:simplepattern}(2), any essential torus neither splits $C_{13}$ and $C_{12} \cup C_{22}$, nor splits $C_{22}$ and $C_{12} \cup C_{13}$. As demonstrated before, by Criterion \ref{prop: throw}, it remains to show $C_{12}$ is simple, which is proved in (i).

For (v), we have shown that $C_{12}$, $C_{21}$ are simple when proving (i), $C_{13}$ and $C_{22}$  are simple when proving (ii), and $C_{31}$ is simple by (iii). For the remained case that $C_{11}$ is simple, take cross disk $D_{13}$ intersecting $C'_{11}$ in 4 points. This is Case (1.0.0) in Lemma \ref{lem: 6cases}(1). Delete the ear $E$ (See Subsection \ref{subsect:leastintersectionpoints}). Then using the approach in Example 5.3 in \cite{BW}, we can show the asserted incredible circle in Criterion \ref{prop:least}(1) does not exist.
\end{example}

\subsection{Interior disks only}\label{subsect:interioronly}
If there is no cross disk, we can give a sufficient and necessary criterion to detect s-primeness, parallel to Theorem \ref{thm:untie}.

\begin{thm}\label{thm:interiorprime}
Let $L=\cup_{i=1}^n C_i$ be an $n$-component Brunnian link, let $I$ and $J$ be disjoint non-empty subsets of $\{1,...,n\}$, and let $D_i ^I$($i \in I$) and $D_j ^J$($j \in J$) be $L_I$-interior disks and $L_J$-interior disks respectively. Assume that each interior disk is s$\vee 8$, and
\begin{center}
$U_I = L_I  \bigcup_{i \in I} D_i ^I$

$U_J = L_J  \bigcup_{j \in J} D_j ^J$
\end{center}
are disjoint disk systems. Then there is no essential torus in the complement of $L$ splitting $L_I$ and $L_J$ if and only if any torus splitting $U_I$ and $U_J$, if exists, is $\partial$-parallel to a component of $L$.
\end{thm}

\begin{proof}
We first prove the ``if'' implication. Suppose that $T$ is an essential torus in the complement of $L$ splitting $L_I$ and $L_J$. Then by Theorem \ref{thm:simplepattern}, $(U_I \cup U_J ) \cap T = \emptyset$. Clearly $T$ is incompressible in the complement of $U_I \cup U_J $. Thus by the condition, $T$ is $\partial$-parallel to a component of $L$, contradicting that $T$ is essential in the complement of $L$.

Now we prove the ``only if'' implication. Suppose that $T$ is torus in the complement of $U_I \cup U_J$ splitting $U_I$ and $U_J$ and not $\partial$-parallel to any component of $L$. Then $T$ splits $L_I$ and $L_J$. By Theorem \ref{thm:sum}, $T$ is incompressible in the complement of $L$. This completes the proof.
\end{proof}

For Brunnian links with at least 3 components, the ``if'' implication of this theorem gives a method to prove s-primeness, while the ``only if'' implication indicates this method is theoretically universal. However, this method does not work for 2-component links, and will become complicated when the number of components is large.

\begin{example}\label{example:W5In}
\begin{figure}[htbp]
		  \centering
    \includegraphics[width=\textwidth]{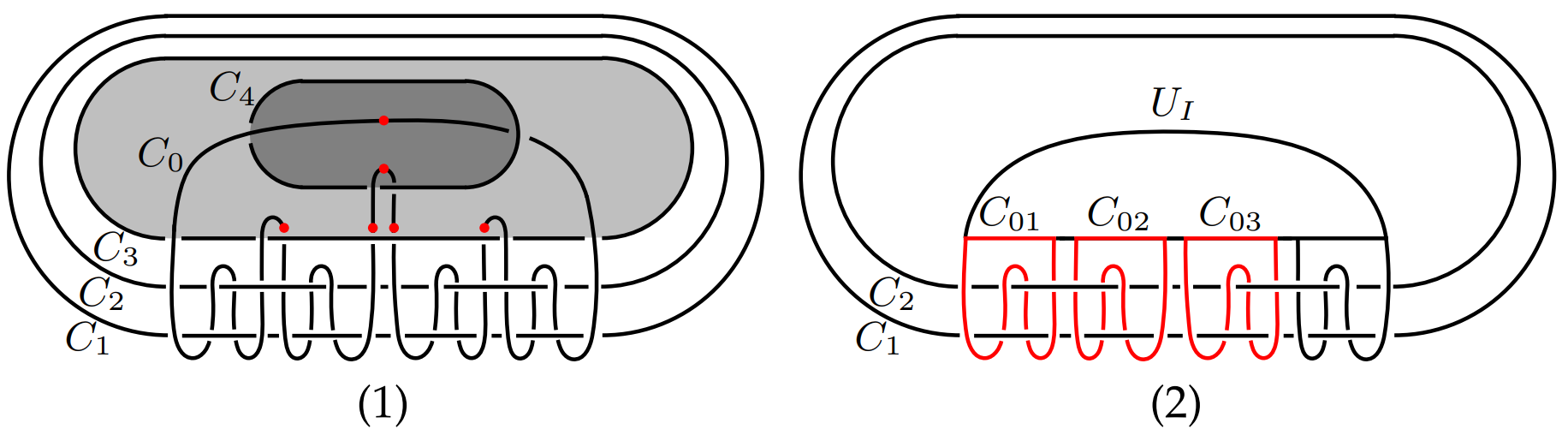}
    \caption{W(5).}
	\label{fig:W5In}
\end{figure}
For the link W($5$) in Fig. \ref{fig:W5In}(1), we show no torus splits $C_1 \cup C_2$ and $C_0 \cup C_3 \cup C_4$. Take interior disks $D_3 ^I$ and $D_4 ^I$ as shown in Fig. \ref{fig:W5In}(1). By Theorem \ref{thm:interiorprime}, we only need to prove no torus splits $C_1 \cup C_2$ and $C_0 \cup D_3 ^I \cup D_4 ^I$. The regular neighbourhood of the later complex is isotopic to the regular neighbourhood of the graph $U_I$ in Fig. \ref{fig:W5In}(2). The result follows from the claim that no torus splits the red link $C_{01} \cup C_{02} \cup C_{03}$ and $C_1 \cup C_2$, which will be proved in the last paragraph of this paper by a generalized version of Theorem \ref{thm:interiorprime}.
\end{example}

For more examples applying this method, see Appendix \ref{subsect:appendix1}.

\subsection{Analysis on plausible annuli}\label{subsect:plausibleannuli}

To show a hypothetical torus cannot exist, we may estimate the shape of the torus. We thus introduce some notations.

Under Standard premise, the components of the intersection of $T$ and the cross disks are \emph{plausible circles}, denoted $\tilde{C}_1,\tilde{C}_2,...,\tilde{C}_t$, successively indexed along $T$. For any $k =1,2,...,t$, the annulus in $T$ between $\tilde{C}_k$ and $\tilde{C}_{k+1}$ is a \emph{plausible annulus}, denoted $\tilde{A}_k$; the disk bounded by $\tilde{C}_k$ in a cross disk is a \emph{plausible disk}, denoted $\tilde{D}_k$; the connected component of $V_J - (\cup_{r \in R} D_r ^E \cup_{s \in S} D_r ^F)$ containing $\tilde{A}_k$ is a \emph{plausible cylinder}, denoted $\tilde{B}_k$.

\begin{subcriterion}\label{prop: arc}
Under Standard premise, let $\alpha$ be an arc component of $L_J - \cup_{r \in r} D_r ^E \cup \cup _{s \in s} D_s ^F$.

(1) $\alpha$ connects two consecutive plausible disks on the corresponding sides of cross disks unless two endpoints of $\alpha$ are on one side of a cross disk.

(2) If $\alpha$ connects two plausible disks $\tilde{D}_k$ and $\tilde{D}_{k+1}$ in one cross disk $D_i$ , and $\alpha$ is \emph{unknotted}, i. e. forming unknot with an arc connecting $\partial \alpha$ in $D_i$, then $\alpha$ is the core of $\tilde{B}_k$.
\end{subcriterion}

Here, a \emph{core} of a plausible cylinder is a trivial arc in the plausible cylinder connecting its two bases.

\begin{proof}
Since (1) is obvious, we need only prove (2). There are two cases depending on whether $\tilde{D}_k$ and $\tilde{D}_{k+1}$ are on the same side of $D_i$. First suppose they are on the same side of $D_i$. Let $\beta_\alpha$ be an arc in $D_i$ connecting $\partial \alpha$ so that $\beta_\alpha$ intersects both $\tilde{C}_k$ and $\tilde{C}_{k+1}$ in exactly one point. Let $D_\alpha$ be a disk bounded by $\alpha \cup \beta_\alpha$. Consider $D_\alpha$ and $D_i$. A standard innermost circle/outermost arc argument shows that we may assume $D_\alpha \cap D_i = \beta_\alpha$. Then consider $D_\alpha$ and $\tilde{A}_k$. By the choice of $\beta_\alpha$, $D_\alpha \cap \tilde{A}_k$ consists of an arc connecting $\tilde{C}_k$ and $\tilde{C}_{k+1}$ and some trivial circles in $\tilde{A}_k$. We can eliminate all the circle components by a  standard innermost circle argument.  Notice that the remained arc component is parallel to the core of $\tilde{B}_k$. Thus $\alpha$ is the core of $\tilde{B}_k$. See Fig. \ref{fig:arc}(1).

\begin{figure}[htbp]
		  \centering
    \includegraphics[width=0.70\textwidth]{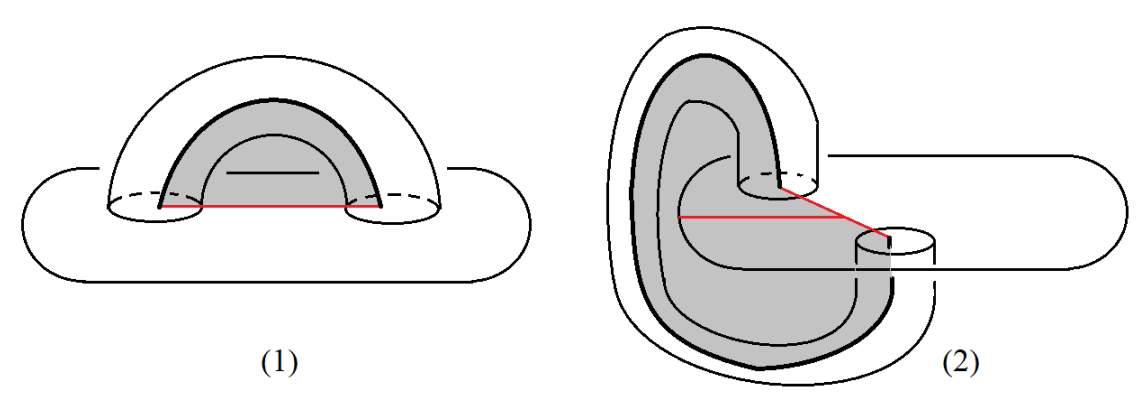}
    \caption{The arc $\alpha$ is the core of $\tilde{B}_k$.}
    \label{fig:arc}
\end{figure}

Now suppose two plausible disks are on different sides of $D_i$. The proof is almost identical, the major change being the substitution of $\beta_\alpha$ into a shape ``T'' graph for $D_\alpha \cap D_i$. See Fig. \ref{fig:arc}(2).
\end{proof}

\begin{example}\label{example:lampprime}
Recall the link in Fig. \ref{fig:lampstable}(1). Using the notations in Example \ref{example:lampstable}, we prove it is s-prime. Suppose there is an essential torus $T$ splitting the link. Take the stable exterior cross disk $D_1$ bounded by $C_1$ as shown in Fig. \ref{fig:lampstable}(2). It is easy to see that $lk(C_1, C_2 )= lk(C_1, \beta _J ) lk( \beta _I,C_2 )$. Since $lk(C_1, C_2 )=0$ and $C_1$ and $C_2$ are symmetric, we may assume $lk(C_1, \beta _J )=0$. Thus there is a plausible cylinder, say $\tilde{B}_k$, whose bases are on the same side of $D_1$. Without loss of generality, assume its bases are on the positive side of $D_1$. Clearly, $\tilde{B}_k$ intersects $C_2$ in some odd labelled arcs.

Let $\alpha_i$ be a component of $\tilde{B}_k \cap C_2$. We claim $\alpha_i$ is the core of $\tilde{B}_k$. In fact, if not, $\partial \alpha_i$ would be in one plausible disk, say $\tilde{D}_k$. Then once we delete all the other odd labelled arcs, $\tilde{A}_k \cup \tilde{D}_{k+1}$ would be a disk capping $\alpha_i$ with $D_1$. This implies $lk(\alpha _i, \alpha _{i + 1})=0$, a contradiction.

Hence all the components of $\tilde{B}_k \cap C_2$ are cores of $\tilde{B}_k$. By Lemma \ref{lem: hopf}(1), $\tilde{B}_k \cap C_2$ has more than one component. Consider two of the components. Then they are parallel to each other in $\tilde{B}_k$. Thus for any even labelled arc $\alpha_j$, this two arcs have the same linking number with $\alpha_j$. This is impossible as we know $lk(\alpha _i, \alpha _{i \pm 1}) = -1$ and otherwise the linking number is $0$.
\end{example}

\subsection{Cross disk with the least intersection points}\label{subsect:leastintersectionpoints}

Recall Proposition \ref{prop:simplepattern}. From its proof we see that if a cross disk $D_i$ intersects $V_J$ in its meridian disks, then $\sharp (D_i \cap L_J ) \geq 4$. In this subsection, we analyse cross disks intersecting $L_J$ in the least number of points in detail.

Let $L=\cup_{i=1}^n C_i$ be an $n$-component Brunnian link. Let $T$ be a torus splitting $S^3$ into two solid tori $V_I$ and $V_J$, such that $L \cap V_I = L_I = \cup_{j \in J} C_i$ and $L \cap V_J = L_J = \cup_{j \in J} C_i$. Let $D_i$ be a cross disk bounded by $C_i$ intersecting $V_J$ in its meridian disks such that $\sharp (D_i \cap L_J )= 4$. We give the following 6 cases on the intersection relationship of $D_i$, $T$ and $L_j$. See Fig. \ref{fig:6cases} for illustration.

If $D_i$ intersects only one component of $L_j$:

$ \begin{array}{llllll}
  Case & lk(C_i, \beta _J ) & lk(C_j, \beta _I ),\ \forall j & Case & lk(C_i, \beta _J ) & lk(C_j, \beta _I ),\ \forall j \\
  (1.0.0): & 0 & 0 & (1.2.0): & 2\ \ (L_I=C_i) & 0 \\
  (1.0.2): & 0 & 2\ \ (L_J = C_j) & (1.2.2): & 2\ \ (L_I = C_i) & 2\ (L_J = C_j )
\end{array}$

If $D_i$ intersects more than one component of $L_j$:

$ \begin{array}{llllll}
  Case & lk(C_i, \beta _J ) & lk(C_j, \beta _I ),\ \forall j & Case & lk(C_i, \beta _J ) & lk(C_j, \beta _I ),\ \forall j \\
  (2.0.0): & 0 & 0 & (2.2.0): & 2\ \ (L_I=C_i) & 0
\end{array}$

\begin{figure}[htbp]
		  \centering
    \includegraphics[width=0.90\textwidth]{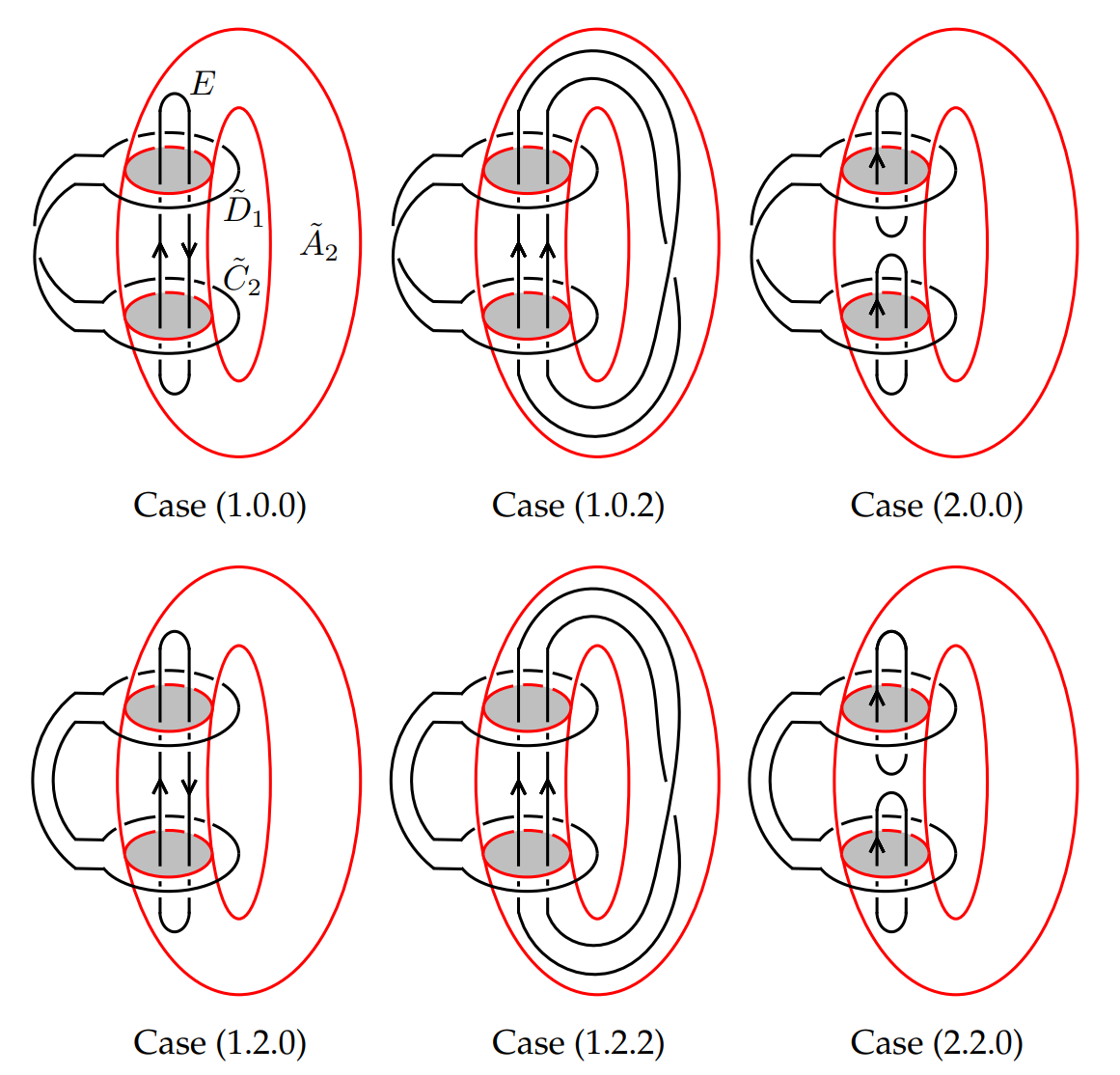}
    \caption{A cross disk with the least intersection points with $L_j$.}
	\label{fig:6cases}
\end{figure}

\begin{lem}\label{lem: 6cases}
(1) If $D_i$ is an exterior cross disk, one of the 6 cases above happens.

(2) If $D_i$ is a free cross disk, one of Case (1.0.0), (1.0.2) and (2.0.0) happens.
\end{lem}

\begin{proof}
On account of linking number and Brunnian property, these are all the allowed cases.
\end{proof}

In the following proposition, we call an arc component of $L_J - D_i$ an \emph{ear} if its endpoints are on the same side of $D_i$.

\begin{subcriterion}\label{prop:least}
Under Standard premise, suppose that there is only one cross disk, denoted $D_i$, and $\sharp (D_i \cap L_J )= 4$.

(1) In Case (1.0.0) (resp. (1.2.0)), there are 2 ears on one side (resp. different sides) of $D_i$ such that after deleting each ear $E$, there is an incredible circle for $U-E$ on the same (resp. the other) side of $D_i$.

(2) In Case (2.0.0) (resp. (2.2.0)), once deleting an ear $E$, there is an incredible circle for $U-E$ on the same (resp. the other) side of $D_i$.
\end{subcriterion}

\begin{proof}
We only give proof for Case (1.0.0), and other cases are similar. Using labels indicated in Fig. \ref{fig:6cases}, after deleting $E$, the disk $\tilde{D_1} \cup \tilde{A_2}$, after a perturbation, shows that $\tilde{C_2}$ is an incredible circle for $U-E$.
\end{proof}

In the following example, we use this criterion in Step 2. Then we use Criterion \ref{prop: arc} in the first paragraph of Step 3. Finally we reduce the problem of showing s-primeness into showing two knotoids are not isotopic.

\begin{example}\label{example:W5prime}
\begin{figure}[htbp]
		  \centering
    \includegraphics[width=\textwidth]{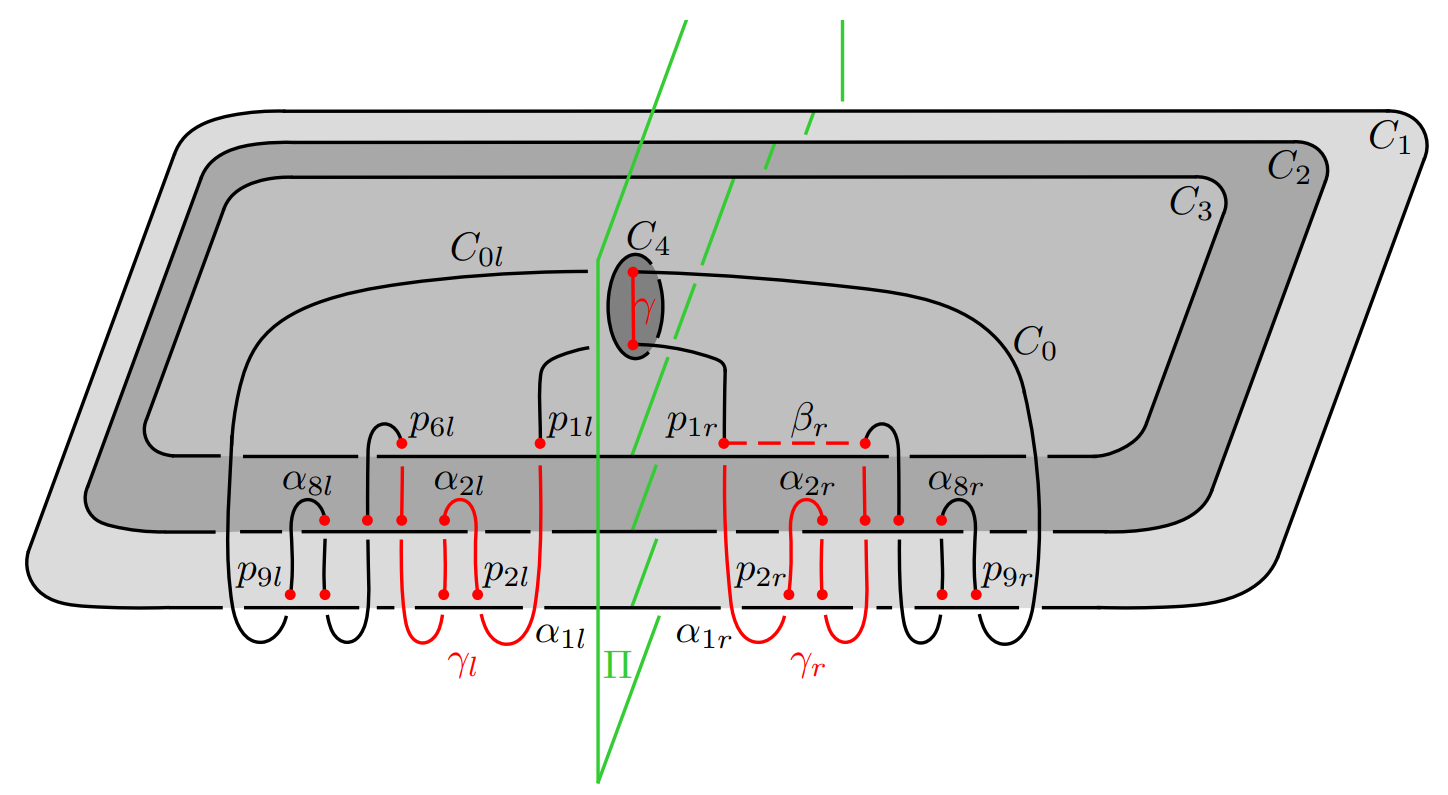}
    \caption{W(5) and disk system.}
	\label{fig:W5prime}
\end{figure}
The link W($5$) is s-prime. Suppose $T$ is an essential splitting torus.

\textbf{Step 1.} Take the disk $D_4$ bounded by $C_4$ as depicted in Fig. \ref{fig:W5prime}. By Proposition \ref{prop:simplepattern}, $C_0 \cup C_4$, also $C_0 \cup D_4$, are at the same side of $T$.

The red arc $\gamma$ in $D_4$ and the left arc in $C_0$ form a circle, say $C_{0l}$, and $C_{0l} \cup C_1 \cup C_2 \cup C_3$ is just W($4$). Notice that W($4$) is a Milnor link, an s-sum of two Borromean rings, as there is a torus splitting $C_{0l} \cup C_3$ and $C_1 \cup C_2$. Thus $T$ either splits $C_0 \cup D_4$ and $C_1 \cup C_2 \cup C_3$, or splits $C_0 \cup D_4 \cup C_3$ and $C_1 \cup C_2$. As we have ruled out the second possibility in Example \ref{example:W5In}, $T$ splits $C_0 \cup C_4$ and $C_1 \cup C_2 \cup C_3$.

\textbf{Step 2.} Take $L_J$-interior disk $D_4 ^J$ and exterior cross disk $D_3 ^E$ as depicted in Fig. \ref{fig:W5prime}. Then $\sharp (D_3 ^E \cap C_0 )= 4$ and this is Case (1.0.0) in Lemma \ref{lem: 6cases}. We claim $p_{1l}, p_{1r}$ are in one plausible disk and $p_{6l}, p_{6r}$ are in another plausible disk, where $\{ p_{1l}, p_{1r}, p_{6l}, p_{6r} \} = D_3 ^E \cap C_0$ as shown in Fig. \ref{fig:W5prime} (in this figure, $ p_{6r} $ is the right endpoint of $ \beta_r $ and we omit its label).

In fact, otherwise once we delete $\gamma_l$, by Criterion \ref{prop:least}(1), $\gamma_r$ would be capped by a disk whose boundary is the asserted incredible circle. Let $\beta_r$ be an arbitrary arc in the plausible disk connecting $\partial \gamma_r$. Then $\gamma_r \cup \beta_r$ is a circle capped by that disk. However, $\gamma_r \cup \beta_r$ is homotopically nontrivial in the complement of $C_1 \cup C_2$, a contradiction.

\textbf{Step 3.} Take $L_J$-interior disk $D_4 ^J$ and exterior cross disks $D_1 ^E$, $D_2 ^E$ and $D_3 ^E$ as depicted in Fig. \ref{fig:W5prime}. Since $p_{1l}$ and $p_{1r}$ are in one plausible disk, by Criterion \ref{prop: arc}(1), $\alpha_{1l}$ and $\alpha_{1r}$ are in one plausible cylinder and $p_{2l}$ and $p_{2r}$ are in one plausible disk. Notice that by Example \ref{example:Wnstable} the disks are all stable. It follows that $p_{il} \cup p_{ir}$ is the intersection of a plausible disk and $C_0$ for any $i=1,2,...,9$, and $\alpha_{il}$ and $\alpha_{ir}$ are in one plausible cylinder, for any $i=1,2,...,8$.

Consider $D_1 ^E \cup D_3 ^E \cup \alpha_{1l} \cup \alpha_{1r}$, then $\alpha_{1l}$ and $\alpha_{1r}$ are parallel. A standard innermost circle argument shows that $\alpha_{1l}$ and $\alpha_{1r}$ are parallel as well in the plausible cylinder. The shape in other plausible cylinders are similar. Consequently, there is a disk $D_0$ bounded by $C_0$ which is contained in $V_J$, intersects each plausible disk in a red arc and intersects $D_4 ^J$ in two green arcs as shown in Fig. \ref{fig:W5VJ}(1).

We see that $D_1 ^E$, $D_2 ^E$, $D_3 ^E$ and $L$ are all symmetric about the middle sphere $\Pi$, and we may assume $D_4 ^J \subset \Pi$. Clearly $D_0 \cap \Pi$ consists of circles and an arc connecting $\partial \gamma$. We can eliminate all the circle components by a standard innermost circle argument so that $D_0 \cap \Pi$ is an arc, say $\gamma_\Pi$, and then isotope $D_0$ so that the shape on $D_0$ of $\gamma_\Pi$ and the red arcs in Fig. \ref{fig:W5VJ}(1) are as shown in Fig. \ref{fig:W5VJ}(2).

\begin{figure}[htbp]
		  \centering
    \includegraphics[width=0.90\textwidth]{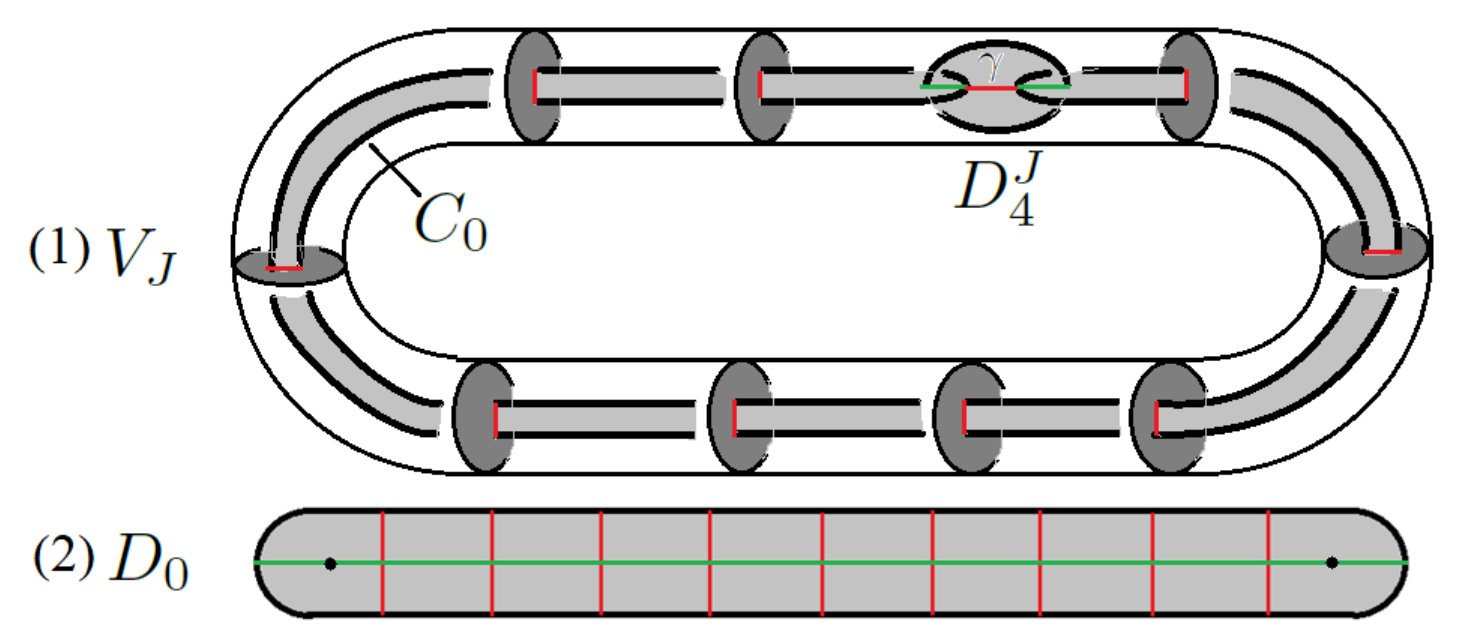}
    \caption{(1) $D_4 ^J \cup D_0$ in $V_J$. (2) The green arc is $\gamma_\Pi$.}
	\label{fig:W5VJ}
\end{figure}

\textbf{Step 4.} The left half of $D_0$ cut by $\gamma_\Pi$ indicates that $C_{0l}$ can be projected to $\Pi$ as a planar curve $\gamma_\cup \gamma_\Pi$, that is, parallel to $\gamma_\cup \gamma_\Pi$ and intersecting the three cross disks in the left halves of the red arcs in Fig. \ref{fig:W5VJ}. To complete the proof, it remains to show this is impossible.

In fact, on account of the order and directions in which $C_{0l}$ passes through the three cross disks, it is straightforward to show that the simple closed curve $\gamma_\cup \gamma_\Pi$ can only be as shown in Fig. \ref{fig:knotoids}(1). On the other hand, $C_{0l}$, projected to $\Pi$, is as shown in Fig. \ref{fig:knotoids}(2). Our problem is reduced to prove these two knotoids are not isotopic. In fact, connecting $a_1, b_2, b_3$ in Fig. \ref{fig:knotoids} forms two Brunnian links. One can verify that the left one is an s-sum of Whitehead link and Borromean rings, while the right one is s-prime, which can be proven by Criterion \ref{prop:least}(1).

\begin{figure}[htbp]
		  \centering
    \includegraphics[width=0.80\textwidth]{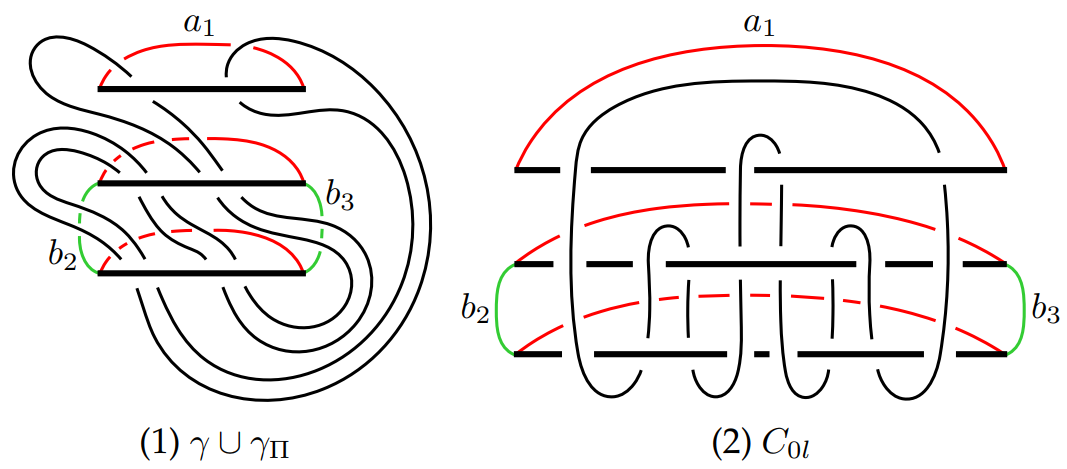}
    \caption{The thick arcs are the projection of left halves of $C_1$, $C_2$ and $C_3$.}
	\label{fig:knotoids}
\end{figure}
\end{example}

\section{Conclusions and generalizations}

\subsection{More hyperbolic Brunnian links.}\label{subsect:examples}

Based on the previous examples, we are now confident that Theorem \ref{thm:series} can be proved.

\begin{proof}[Proof of Theorem \ref{thm:series}]
We have pointed out that all these links are untied in Subsection \ref{subsect:untiednessthm}. We now prove they are s-prime.

(1) The proof in Example \ref{example:debrunner} works without loss of generality for $n \ge 5$. For $n =3$, the proof is just simpler since it suffices to prove every component is simple. For deBrunner($2$), use Case (1.2.0) in Proposition \ref{prop:least}(1). For deBrunner($4$), the case (v) in the proof of Example \ref{example:debrunner} needs to be modified by using Proposition \ref{prop:least}(1) and then the approaches in \cite{BW} to show there is no incredible circle.

(2) A proof of s-primeness will follow Example \ref{example:W5prime} by induction on $n$. We know W($5$) is s-prime. For any $n>5$, W($n-1$) is s-prime, an argument similar to the one used in the first paragraph in Example \ref{example:W5prime} shows that the only possibility is that $T$ splits $C_0 \cup D_n$ from all the other components of the link. The remainder of the proof is identical, except that in the last paragraph we distinguish two knotoids in a simpler manner as follows. For each thick arc, connecting the red arcs from below as shown in Fig. \ref{fig:knotoids}, we get two links. The left one will be a Milnor link, which is an s-sum\cite{BM}, while the right one will be W($n-1$), which is s-prime by induction.

(3) For any $m>1,n>2$, we have shown that Torus($m,n$) is s-prime in Example \ref{example:torus(m,n)}. For $m>1,n=2$, the proof is analogues to that when $m>1,n>2$ and is left to the reader. For $n=1$, the link Torus($m,1$) is isotopic to deBrunner($m$). For $m=1,n>2$, Torus($1,n$) is Brunn's chain (see \cite{Brunn, BW} or recall Example \ref{example:brunnlink}). The proof of that Brunn's chains are s-prime is quite similar to that for deBrunner($n$) and so is omitted.

(4) By suitable modification to the proof of that Torus($m,n$) and Carpet($m,n,p$) are s-prime, we can show  Tube($m,n$) is s-prime for $m>0,n>1$. We leave the details to the reader.

(5) When $n=1$, Carpet($m,n,p$) is the $p$-component Brunn's chain. When $n>1$, the proof in Example \ref{example:carpet} shows Carpet($m,n,p$) are all s-prime without loss of generality.

(6) The proof in Example \ref{example:lampprime} shows the links in this family are all s-prime without loss of generality.
\end{proof}

We detect s-primeness for all the Brunnian links in literature. We know far more Brunnian links are s-prime, including
\begin{enumerate}
\item Tait series (Fig. 2 in \cite{J}),
\item Baas' solids (Fig. 16 in \cite{BCS}),
\item Snakes (\cite{BW}),
\item Cirrus (\cite{BW}),
\item Wheels (\cite{BW}) 
\end{enumerate}
and so on. As they are also untied as explained in Subsection \ref{subsect:untiednessthm}, combining Theorem \ref{thm:classification0}, we have
\begin{thm}
The five families of links above are hyperbolic. $ \square $
\end{thm}

\subsection{Generalizations}\label{subsect: generallinks}

While we have considered only Brunnian links in this paper, the methods can be extended to links having some unknotted components. 

Let us clarify some basic facts. A link $L$ in $S^3$ is \emph{unsplit} if no sphere in the complement space splits $L$. Let $T$ be an essential torus in the complement space. If $T$ does not split $L$, then $T$ must be knotted and $L$ is contained in the solid torus bounded by $T$ in a geometrically essential way. If $T$ splits $L$, then $T$ can be either knotted or unknotted.

\subsubsection{Hyperbolicity of links with some unknotted components.}

Replacing the conditions of disks all by ``stable'', Theorem \ref{thm:untie}, Theorem \ref{thm:interiorprime} and Theorem \ref{thm:simplepattern} hold for more general links.

\begin{thm}\label{thm:unsplit}
(Decision theorem for unsplitting essential torus)

Let $L$ be an unsplit link so that $C_i$ is an unknot component of $L$ for each $1 \le i \le k$. Let $D_i$ be a stable disk bounded by $C_i$ for each $1 \le i \le k$ so that $U = \bigcup_{i=1}^k D_i \cup L$ is a disk system. Then there is no unsplitting essential torus in  $S^3 -L$ if and only if there is no incompressible knotted torus in $S^3 -U$.
\end{thm}

\begin{proof}
The proof is similar but only simpler as that of Theorem \ref{thm:untie}. 
\end{proof}

\begin{thm}\label{thm:split}
(Criterion for splitting essential torus)

Let $L_I$ and $L_J$ be two proper sublinks of an unsplit link $L$ with $L_I \sqcup L_J = L$, let $C_i(i \in I_0)$ and $C_j(j \in J_0)$ be some unknot components of $L_I$ and $L_J$ respectively and let  $D_i ^I$($i \in I_0$) and $D_j ^J$($j \in J_0$) be stable disks bounded by $C_i ^I$($i \in I_0$) and $C_j ^J$($j \in J_0$). Assume that
\begin{center}
$U_I = L_I  \bigcup_{i \in I} D_i ^I$

$U_J = L_J  \bigcup_{j \in J} D_j ^J$
\end{center}
are disjoint disk systems, and assume that any torus splitting $U_I \cup U_J$, if exists, is $\partial$-parallel to a component of $L$. Then there is no essential torus splitting $L_I$ and $L_J$.
\end{thm}

\begin{proof}
This theorem can be proved in a similar way as in the proof of Theorem \ref{thm:interiorprime} together with Step 1 and Step 2 in the proof of Theorem \ref{thm:simplepattern}. The major change is that we need to replace torus $T$ instead of isotoping it in Step 1.2 and Step 2.2.
\end{proof}

We can still utilize cross disks as follows.

\begin{thm}\label{thm:simplepattern2}
(Simple intersection pattern theorem for cross disks)

Let $T$ be an essential torus in the complement of an unsplit link $L$. Let $V$ be a solid torus bounded by $T$ in $S^3$, so that $L \cap V$ and $L \cap X = L_I = \cup_{i \in I} C_i$  are proper sublinks of $L$, where $X = S^3 - int V$. Let $Q, R, S$ be mutually disjoint subsets of $I$, and let 
$D_q ^I$($q \in Q$), $D_r ^E$($r \in R$) and $D_s ^F$($s \in S$) be $L_I$-interior disks, exterior cross disks, and free cross disks respectively. 

Assume that each disk is stable, the cross disks are mutually disjoint, and the union of these disks and $L$ forms a disk system.

Then $L$ is split by an essential torus which each disk intersects in simple intersection pattern.
\end{thm}

\begin{proof}
This theorem can be proved in a similar way as in Step 1 and Step 3 in the proof of Theorem \ref{thm:simplepattern}. The major change is that we need to replace torus $T$ instead of isotoping it in Step 1.2 and Step 3.2. 
\end{proof}

Recall from \cite{Th} that if an unsplit link has atoroidal complement space and is not Seifert, then the link is hyperbolic. Using the above theorems, we can give quick proofs for hyperbolicity of many links in Thurston's book \cite{Th}, and give a simpler proof for that the links in \cite{Lee} Section 5 (Fig. 2) are hyperbolic.

\subsubsection{Using spanning surfaces instead of spanning disks.}

It is worth to explore how to generalize our methods by using spanning surfaces for some components instead of just using spanning disks. Here we only give an example of using spanning annuli.

To show no torus splits $C_{01} \cup C_{02} \cup C_{03}$ and $C_1 \cup C_2$ in Fig. \ref{fig:generalization}, we may take the annulus $A$ bounded by $C_{01} \cup C_{02}$, intersecting $C_{03}$ in 4 points. Firstly, we can prove there is no incredible circle on $A$. Then if such a torus exists, say $T$, which bounds a solid torus  $V$ containing $C_{01} \cup C_{02} \cup C_{03}$, we can prove $A \subset V$ after isotopy. In fact, we can eliminate inessential circles in $T$ as demonstrated before. To eliminate annulus regions in $A$ outside $V$, the main point of our argument is that for any $i=1,2,3$, $C_{0i} \cup C_1 \cup C_2$ is a Borromean rings, which is atoroidal, and thus $C_{01}$ and $C_{02}$ can only be cores of $V$, and such annulus regions in $A$ can only be $\partial$-parallel to $T$. Finally, it is easy to see no torus splits $C_1 \cup C_2$ and $A \cup C_{03}$ by considering a deformation retraction of $A \cup C_{03}$.

\begin{figure}[htbp]
		  \centering
    \includegraphics[width=0.60\textwidth]{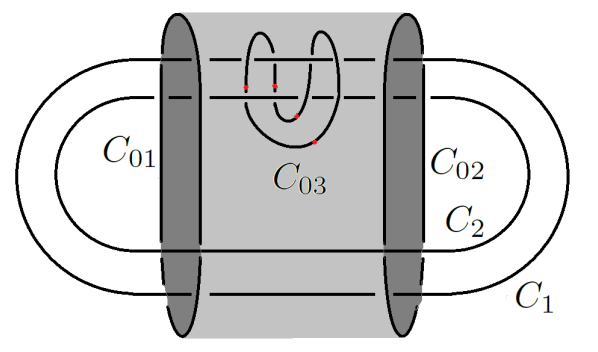}
    \caption{No torus splits $C_{01} \cup C_{02} \cup C_{03}$ and $C_1 \cup C_2$.}
	\label{fig:generalization}
\end{figure}

\textbf{Acknowledgement}:
The author is grateful to Professor Jiming Ma for helpful discussions and generous supports. The author thanks Weibiao Wang, who drew most of the figures in this paper and verified that Cirrus (Fig. \ref{fig:cirrus}) is hyperbolic. The author sincerely thanks an anonymous referee for carefully reading the full text and for correcting many presentational problems in a previous version.

\bibliographystyle{amsalpha}

\section{Appendix}

\subsection{Tait series}\label{subsect:appendix1}

We provide an alternative proof for that a series of alternating Brunnian links are s-prime to illustrate the method in Subsection \ref{subsect:interioronly}.

\begin{example}

\begin{figure}[htbp]
		  \centering
    \includegraphics[scale=0.35]{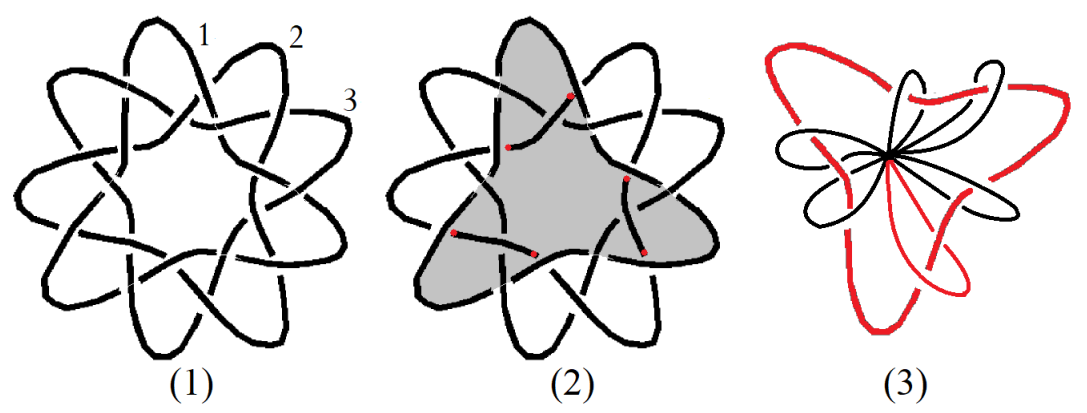}
    \caption{A link in Tait series (Fig. 2 in \cite{J}).}
	\label{fig:borromean2}
\end{figure}

Fig. \ref{fig:borromean2}(1) shows a link in a Tait series. In view of symmetry, we only need to prove no essential torus splits $C_3$ from $C_1$ and $C_2$. Take the disk $D_1$ as shown in Fig. \ref{fig:borromean2}(2), which is stable by an argument similar to the one used in Example \ref{example:lampstable}. Then a deformation retraction of $D_1 \cup C_2 $ is as shown in Fig. \ref{fig:borromean2}(3). In view of the complement space of the red part, we see that the only splitting torus is $\partial$-parallel to $C_3$.
\end{example}

\subsection{Hyperbolicity of Cirrus}\label{subsect:appendix2}
We show that Cirrus in Fig. \ref{fig:cirrus} is untied. That Snakes and Wheels are untied can be proved similarly.

Take the union of credible disks depicted in Fig. \ref{fig:Cirruscomplex} to form a disk system $U$. The sequence of isotopies of $U$ illustrated in Fig. \ref{fig:Cirrusgraph} shows that $S^3 - int N(U)$ is a handlebody, and thus atoroidal. By Theorem \ref{thm:untie}, Cirrus is untied, and thus hyperbolic.
\begin{figure}[htbp]
		  \centering
    \includegraphics[width=\textwidth]{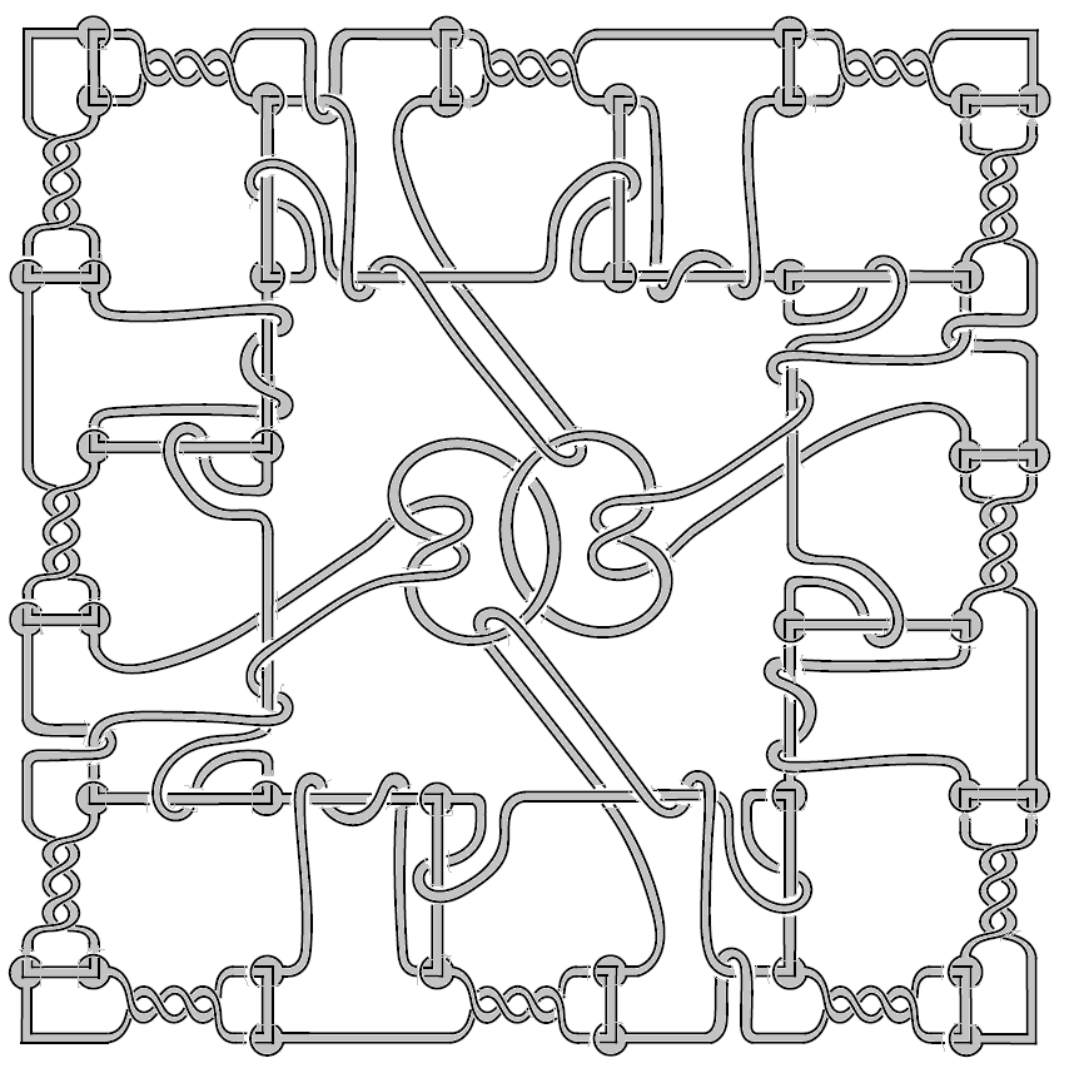}
    \caption{A disk system for Cirrus.}
	\label{fig:Cirruscomplex}
\end{figure}

\begin{figure}[htbp]
		  \centering
    \includegraphics[scale=0.65]{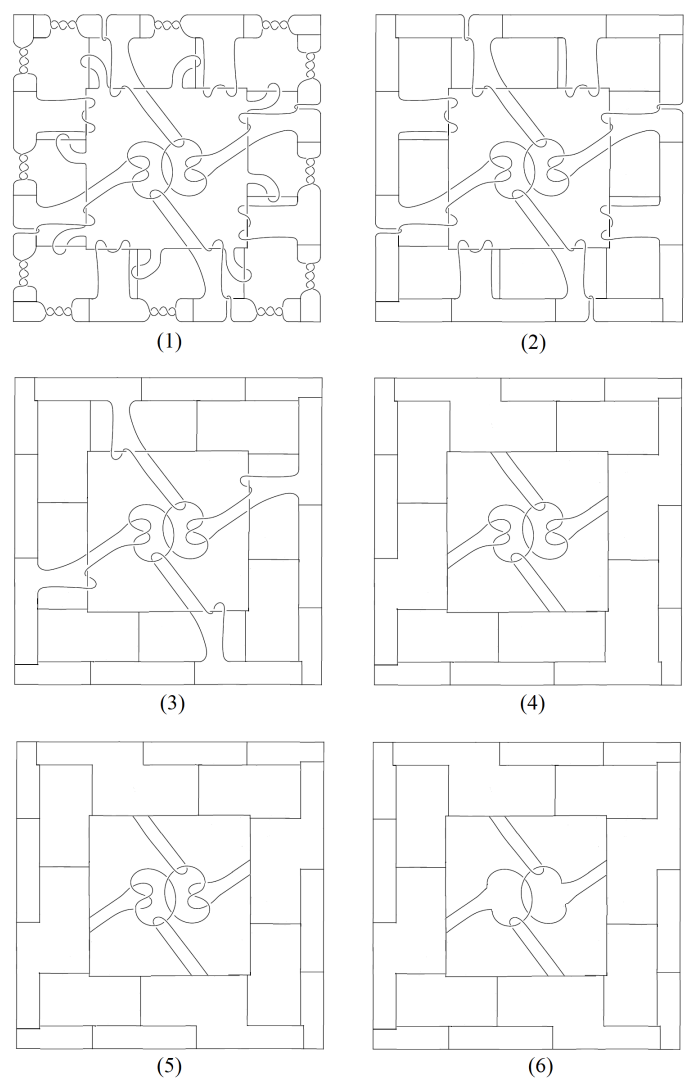}
    \caption{A sequence of transformations for a deformation retraction graph of $U$.}
	\label{fig:Cirrusgraph}
\end{figure}

\end{document}